\documentclass[12pt,a4paper]{amsart}
\usepackage{amssymb}
\usepackage{amsfonts}
\usepackage{amsmath}
\usepackage[mathscr]{eucal}

\usepackage{comment}

\newtheorem{thm}{Theorem}[section]
\newtheorem{prop}[thm]{Proposition}
\newtheorem{lem}[thm]{Lemma}
\newtheorem{cor}[thm]{Corollary}

\newtheorem{rem}[thm]{Remark}
\newtheorem{conj}[thm]{Conjecture}
\numberwithin{equation}{section}

\def\fa{{\mathfrak a}}
\def\fb{{\mathfrak b}}

\def\ff{{\mathfrak f}}

\def\fm{{\mathfrak m}}
\def\fn{{\mathfrak n}}
\def\fo{{\mathfrak o}}
\def\fp{{\mathfrak p}}

\def\fc{{\mathfrak c}}

\def\fP{{\mathfrak P}}

\def\GL{{\operatorname {GL}}}

\def\SL{{\operatorname{SL}}}
\def\SO{{\operatorname{SO}}}
\def\U{{\operatorname {U}}}

\def\PGL{{\operatorname{PGL}}}

\def\Re{{\operatorname {Re}}}

\def\nr{{\operatorname{N}}}

\def\sgn{{\operatorname {sgn}}}

\def\leq{\leqslant}
\def\geq{\geqslant}
\def\bsl{\backslash}
\def\le{\leq}
\def\ge{\geq}

\def\JJ{\mathbb{J}}

\def\1{{\bold 1}}

\renewcommand{\a}{\alpha}
\renewcommand{\b}{\beta}

\newcommand{\e}{\epsilon}

\renewcommand{\l}{\lambda}

\newcommand{\s}{\sigma}

\newcommand{\ga}{{\mathfrak{a}}}
\newcommand{\gb}{{\mathfrak{b}}}
\newcommand{\gc}{{\mathfrak{c}}}
\newcommand{\gf}{{\mathfrak{f}}}

\newcommand{\gn}{{\mathfrak{n}}}
\newcommand{\go}{{\mathfrak{o}}}
\newcommand{\gp}{{\mathfrak{p}}}
\newcommand{\gq}{{\mathfrak{q}}}

\newcommand{\Acal}{{\mathcal A}}
\newcommand{\Bcal}{{\mathcal B}}

\newcommand{\Ecal}{{\mathcal E}}
\newcommand{\Fcal}{{\mathcal F}}

\newcommand{\Ical}{{\mathcal I}}

\newcommand{\Ocal}{{\mathcal O}}

\newcommand{\Scal}{{\mathcal S}}

\renewcommand{\AA}{\mathbb{A}}

\newcommand{\CC}{\mathbb{C}}

\newcommand{\II}{\mathbb{I}}
\newcommand{\LL}{\mathbb{L}}
\newcommand{\NN}{\mathbb{N}}

\newcommand{\QQ}{\mathbb{Q}}
\newcommand{\RR}{\mathbb{R}}

\newcommand{\ZZ}{\mathbb{Z}}

\newcommand{\bfc}{{\mathbf c}}

\newcommand{\bfs}{{\mathbf s}}

\newcommand{\bfx}{{\mathbf x}}

\newcommand{\bfB}{{\mathbf B}}

\newcommand{\bfK}{{\mathbf K}}

\newcommand{\notdivide}{\nmid}

\newcommand{\ord}{\operatorname{ord}}

\newcommand{\Sp}{\operatorname{Sp}}
\newcommand{\fin}{{\rm fin}}
\renewcommand{\Re}{\operatorname{Re}}

\newcommand{\Res}{\operatorname{Res}}
\newcommand{\Sym}{\operatorname{Sym}}

\newcommand{\Nice}{\operatorname{Nice}}

\title{Low-lying zeros of symmetric power $L$-functions
weighted by symmetric square $L$-values}

\author{Shingo Sugiyama}
\address{Department of Mathematics, College of Science and Technology, Nihon University, Suruga-Dai, Kanda, Chiyoda, Tokyo 101-8308, Japan}
\email{sugiyama.shingo@nihon-u.ac.jp}



\begin{document}

\begin{abstract}
		For a totally real number field $F$ and its ad\`ele ring $\AA_F$,
		let $\pi$ vary in the set of irreducible cuspidal automorphic representations
		of $\PGL_2(\AA_F)$ corresponding to primitive Hilbert modular forms of a fixed weight.
		Then, we determine the symmetry type of the one-level density of low-lying zeros of the symmetric power $L$-functions $L(s,\Sym^r(\pi))$ weighted by special values of symmetric square $L$-functions $L(\frac{z+1}{2},\Sym^2(\pi))$ at $z \in [0, 1]$ in the level aspect.
		If $0 < z \le 1$, our weighted density in the level aspect has the same symmetry type as Ricotta and Royer's density of low-lying zeros of symmetric power $L$-functions for $F=\QQ$ with harmonic weight.
		Hence our result is regarded as a $z$-interpolation of Ricotta and Royer's result.
If $z=0$, density of low-lying zeros weighted by central values is a different type
only when $r=2$,
and it does not appear in random matrix theory as Katz and Sarnak predicted.
Moreover, we propose a conjecture on weighted density of low-lying zeros of $L$-functions
by special $L$-values.

In the latter part, Appendices \ref{Comparison of the ST trace formula with Zagier's formula}, \ref{Comparison of the ST trace formula with Takase's formula} and \ref{Comparison of the ST trace formula with Mizumoto's formula} are dedicated to the comparison among several generalizations of Zagier's parameterized trace formula.
We prove that the explicit Jacquet-Zagier type trace formula (the ST trace formula) by Tsuzuki and the author
recovers all of Zagier's, Takase's and Mizumoto's formulas by specializing several data.
Such comparison is not so straightforward and includes non-trivial analytic evaluations.
\end{abstract}

\maketitle

\section{Introduction}

To study zeros of $L$-functions is one of principal problems in number theory
as was originated by the Riemann hypothesis.
However, as of now, investigation of individual $L$-functions is still far from completion.
Instead of individual $L$-functions,
a family of $L$-functions is more tractable.
Concerning zeros of $L$-functions in a family,
Katz and Sarnak \cite{KS1}, \cite{KS2} suggested a philosophy
called {\it the Density Conjecture}
that
the distribution of low-lying zeros of a family of $L$-functions
should have a density function with symmetry type arising in random matrix theory.
Due to their philosophy, it is expected that
the following five density functions should describe density of low-lying zeros of $L$-functions in a family:
\[W(\Sp)(x)=1-\frac{\sin 2\pi x}{2\pi x},\]
\[W({\rm O})(x) = 1+\frac{1}{2}\delta_0(x),\]
\[W(\SO({\rm even}))(x) = 1+\frac{\sin 2\pi x}{2\pi x},\]
\[W(\SO({\rm odd}))(x) = 1-\frac{\sin 2\pi x}{2\pi x} +\delta_0(x),\]
\[W(\U)(x)=1,\]
where $\delta_0$ is the Dirac delta distribution supported at $0$.
The philosophy was derived from their work \cite{KS1} and \cite{KS2} on funtion field cases,
which study statistics of zeros for a geometric family of zeta functions for a function field over a finite field.

Later, Iwaniec, Luo and Sarnak \cite{ILS} gave densities of low-lying zeros of the standard $L$-functions and those of symmetric square $L$-functions associated with holomorphic elliptic cusp forms both in the weight aspect and the level aspect, assuming GRH of several $L$-functions.
Inspired by their study, densities of low-lying zeros of families of automorphic $L$-functions have been investigated in several settings such as Hilbert modular forms (\cite{LiuMiller}),
Siegel modular forms of degree $2$ (\cite{KWY1}, \cite{KWY2}),
and Hecke-Maass forms (\cite{AlpogeAILMZ}, \cite{AlpogeMiller} \cite{GoldfeldKontorovich},
\cite{LiuQi}, \cite{MatzTemplier}, \cite{RamacherWakatsuki}).
As of now, the broadest setting for low-lying zeros
of automorphic $L$-functions was investigated 
by Shin and Templier \cite{ShinTemplier}.
They treated both the weight aspect and the level aspect of low-lying zeros
of automorphic $L$-functions $L(s,\pi, r)$, where $\pi$ varies in discrete
automorphic representations of $G(\AA_F)$ where $G$ is a connected reductive group over a number field $F$ such that $G$ admits discrete series representations at all archimedean places, and $r : {}^L G \rightarrow \GL_d(\CC)$ is an irreducible $L$-morphism
under the hypothesis on the Langlands functoriality principle for $r$.
Before their study,
G\"ulo\u{g}lu \cite{Guloglu} and Ricotta and Royer \cite{RicottaRoyer}
had considered the symmetric tensor representation $\Sym^r :{}^L\PGL_2 \rightarrow \GL_{r+1}(\CC)$ for $r \in \NN$ as a functorial lifting
and gave densities of low-lying zeros of the symmetric power $L$-functions $L(s, \Sym^r(f))$
attached to holomorphic elliptic cusp forms $f$ in the weight aspect \cite{Guloglu} with GRH and in the level aspect \cite{RicottaRoyer} without GRH, respectively,
under the hypothesis on analytic properties of $L(s, \Sym^r(f))$.

Recently, Knightly and Reno \cite{KnightlyReno} studied density of low-lying zeros of the standard $L$-functions $L(s,f)$ attached to holomorphic elliptic cusp forms $f$ weighted by central values $L(1/2,f)$, and found
the change of the symmetry type of the density
from $W({\rm O})$ in the usual setting to $W({\rm Sp})$ in the weighted setting.
Their study was inspired by
the works of Kowalski, Saha and Tsimerman \cite{KowalskiST}
and of Dickson \cite{Dickson} on the asymptotic formula of the average of spinor $L$-functions $L(s, f, {\rm Spin})$
attached to holomorphic Siegel cusp forms $f$ of degree $2$ weighed by the square of the Fourier coefficient of $f$ at the unit matrix.
In \cite{KowalskiST} and \cite{Dickson}, they observed a phenomenon of changing the symmetry type 
from $W({\rm O})$ to $W({\rm Sp})$, which they considered as a weak evidence toward B\"ocherer's conjecture which was not proved at that time (Now this conjecture is known to be true by Furusawa and Morimoto \cite[Theorem 2]{FurusawaMorimoto}).
Indeed, they considered low-lying zeros of $L(s,f, {\rm Spin})$
attached to holomorphic Siegel cusp forms $f$ on ${\rm GSp}_4$
weighted by the Fourier coefficient of each $f$ at the unit matrix (a Bessel period of $f$),
where we remark that
B\"ocherer's conjecture shows that
the Bessel period of $f$ used there is identical to $L(1/2, f, {\rm Spin}) L(1/2, f\otimes\chi_{-4}, {\rm Spin})$,
where $\chi_{-4}$ is the quadratic Dirichlet character modulo $4$.

In this article, in order to observe phenomena of changes of densities by special $L$-values in other settings,
we consider low-lying zeros of the symmetric power $L$-functions $L(s, \Sym^r(\pi))$
associated with cuspidal representations $\pi$ of $\PGL_2(\AA_F)$
corresponding to Hilbert cusp forms over a totally real number field $F$ in the level aspect without GRH,
where our weight factors are
special values of symmetric square $L$-functions $L(\frac{z+1}{2}, \Sym^2(\pi))$
at each $z \in [0, 1]$.

\subsection{Density of low-lying zeros weighted by symmetric square $L$-functions}
\label{Densities weighted by special values of symmetric square $L$-functions}
We prepare some notions.
Let $F$ be a totally real number field with $d_F=[F:\QQ]<\infty$ and $\AA=\AA_F$ the ad\`ele ring of $F$.
Let $\Sigma_\infty$ denote the set of the archimedean places of $F$.
Let $l=(l_v)_{v \in \Sigma_\infty}$ be a family of positive even integers
and let $\gq$ be a non-zero prime ideal of the integer ring $\go$ of $F$.
We denote by $\Pi_{\rm cus}^*(l, \gq)$ the set of all irreducible cuspidal automorphic representations $\pi=\otimes_v'\pi_v$ of $\PGL_2(\AA)$ such that
the conductor of $\pi$ equals $\gq$ and $\pi_v$ for each $v\in \Sigma_\infty$
is isomorphic to the discrete series representation of $\PGL_2(\RR)$ with minimal ${\rm O}(2)$-type $l_v$.
For any $r \in \NN$, we treat in this article the symmetric power $L$-functions $L(s, \Sym^r(\pi))$ for $\pi \in \Pi_{\rm cus}^*(l,\gq)$ explained in \S \ref{Symmetric power $L$-functions} (see also \cite{CogdellMichel} and \cite{RicottaRoyer}), and consider
the hypothesis $\Nice(\pi,r)$ consisting of the following:
\begin{itemize}
	
	\item $L(s, \Sym^r(\pi))$ is continued to an entire function on $\CC$
	of order $1$.
	
	\item $L(s,\Sym^r(\pi))$ satisfies the functional equation
	\[L(s,\Sym^r(\pi))=\e_{\pi,r}(D_F^{r+1} \nr(\gq)^r)^{1/2-s}L(1-s,\Sym^{r}(\pi)),\]
	where $D_F$ is the absoulte value of the discriminant of $F/\QQ$,
	$\nr(\gq)$ is the absolute norm of $\gq$, and
	$\e_{\pi,r}\in \{\pm 1\}$.

\end{itemize}
This hypothesis is expected to be true in the view point of the Langlands functoriality principle.
In the case $r=1$, this hypothesis is well-known to be true.
Since any $\pi \in \Pi_{\rm cus}^*(l,\gq)$ has non-CM,
the hypothesis $\Nice(\pi,r)$ is known for $r=2$ by Gelbart and Jacquet
\cite[(9.3) Theorem]{GelbartJacquet},
$r=3$ by Kim and Shahidi \cite[Corollary 6.4]{KimShahidi} and $r=4$ by Kim \cite[Theorem B and \S7.2]{Kim}.
Recently,
$\Nice(\pi,r)$ was proved for all $r\in \NN$, all $\pi \in \Pi_{\rm cus}^*(l,\gq)$ and all non-zero prime ideals $\gq$ if we restrict our case to elliptic modular forms ($F=\QQ$).
Indeed, $\Sym^r(\pi)$ is an irreducible $C$-algebraic cuspidal automorphic representation
of $\GL_{r+1}(\AA_\QQ)$ by Newton and Thorne \cite{NewtonThorne} which is applied to
the case where the conductor of $\pi$ is square-free.
They also treated in \cite{NewtonThorne2} the case of general levels of elliptic modular forms.
We survey known results on meromorphy of $L(s, \Sym^r(\pi))$ and on
automorphy of $\Sym^r(\pi)$ in \S \ref{Symmetric power $L$-functions}.

Throughout this article, we fix $F$ and $l$, and assume $\Nice(\pi,r)$ for all non-zero prime ideals $\gq$ and all $\pi \in \Pi_{\rm cus}^*(l,\gq)$.
In what follows, we consider distributions of low-lying zeros of $L(s,\Sym^r(\pi))$
for $\pi \in \Pi_{\rm cus}^*(l,\gq)$ with $\nr(\gq)\rightarrow \infty$
without assuming GRH of any $L$-functions.
The one-level density of such low-lying zeros is defined as
\begin{align}\label{def one-level density}
D(\Sym^r(\pi),\phi)=\sum_{\rho=1/2+i\gamma} \phi\left(\frac{\log Q(\Sym^r(\pi))}{2\pi}\gamma\right)\end{align}
for any Paley-Wiener functions $\phi$, where $\rho=1/2+i\gamma$ ($\gamma \in \CC$) runs over all zeros of $L(s,\Sym^r(\pi))$ counted with multiplicity, and $Q(\Sym^r(\pi))$ is the analytic conductor of $\Sym^r(\pi)$.
Here a Payley-Wiener function is given by
a Schwartz function $\phi$ on $\RR$
such that the Fourier transform
\[\hat\phi(\xi)=\int_{\RR}\phi(x)e^{-2\pi i x\xi}dx\]
of $\phi$ has a compact support. By the compactness, $\phi$ is extended to an entire function on $\CC$.
By the symmetry of zeros of $L(s, \Sym^r(\pi))$, we may assume that $\phi$ is even.

Before stating our results, we introduce Ricotta and Royer's work \cite{RicottaRoyer} on elliptic modular forms.
When $F=\QQ$, $\Pi_{\rm cus}^*(l,q\ZZ)$ for an even positive integer $l$ and a prime number $q$ is identified with the set $H_{l}^*(q)$ of normalized new Hecke eigenforms in the space $S_l(\Gamma_0(q))^{\rm new}$ of elliptic cuspidal new forms of weight $l$ and level $q$ with trivial nebentypus.
Set
\[\omega_q(f) = \frac{\Gamma(l-1)}{(4\pi)^{l-1}\|f\|^2},\]
where $\|f\|$ denotes the Petersson norm of $f$ as in \cite[\S2.1.1]{RicottaRoyer}.
We denote by $\e_{f,r}$ the sign of the functional equation of $L(s,\Sym^r(f))$ for $f \in H_l^*(q)$.

Then, Ricotta and Royer \cite{RicottaRoyer} proved the following in the level aspect
without GRH.
\begin{thm} \cite[Theorems A and B]{RicottaRoyer} \label{RR-AB}
	Let $r$ be any positive integer.
Let $l\ge 2$ be a positive even integer and let $q$ vary in the set of prime numbers.
Assume $\Nice(f,r)$ for all $q$ and all $f \in H_{l}^*(q)$.
Let $\phi$ be an even Schwartz function on $\RR$.
Set
\[\b_1=\left(1-\frac{1}{2(l-2\theta)}\right)\frac{2}{r^2},\]
where $\theta \in [0,1/2)$ is the exponent toward the generalized Ramanujan-Petersson conjecture for $\GL_2$ {\rm (}cf. \cite[\S3.1]{RicottaRoyer}{\rm )}.
If ${\rm supp}(\hat\phi)\subset (-\b_1,\b_1)$, then we have
\[
\lim_{q\rightarrow \infty}\sum_{f \in H_{l}^*(q)}\omega_q(f)D(\Sym^r(f),\phi)
= \begin{cases}\int_\RR \phi(x)W({\rm Sp})(x)dx & (r \text{ is even}),\\
\int_\RR \phi(x)W({\rm O})(x)dx & (r \text{ is odd}).
\end{cases}\]
\end{thm}
Here the weight factor $\omega_q(f)$ is called the harmonic weight,
and would be removable and negligible (cf.\ \cite[Lemma 2.18]{Guloglu} under GRH).
As an analogous result to Ricotta and Royer as above,
we give the density of low-lying zeros of $L(s,\Sym^r(\pi))$ weighted by special values $L(\frac{z+1}{2}, \Sym^2(\pi))$ of symmetric square $L$-functions
at each $z \in [0,1]$.
Now we review the ensemble of the special values $\{L(\frac{z+1}{2}, \Sym^2(\pi))\}_{\pi \in \Pi_{\rm cus}^*(l, \gq)}$.
The value $L(\frac{z+1}{2}, \Sym^2(\pi))$ is believed to be non-negative due to GRH.
Although the non-negativity is still open, it is supported by the asymptotics
\[\sum_{\pi \in \Pi_{\rm cus}^*(l,\gq)}\frac{L(\frac{z+1}{2},\Sym^2(\pi))}{L(1,\Sym^2(\pi))} \sim C\,\nr(\gq)(\log \nr(\gq))^{\delta(z=0)},\qquad \nr(\gq)\rightarrow \infty\]
for some explicit constant $C>0$, which follows from the proof of Tsuzuki and the author's result \cite[Theorem 1.3]{SugiyamaTsuzuki2018} (see also Theorem \ref{quantitative ver of equidist}).
In particular, this average is non-zero for any non-zero prime ideals $\gq$ of $\go$ such that $\nr(\gq)$ is sufficiently large.
Our main result in the level aspect without GRH is stated as follows.
\begin{thm}\label{main: weighted one level density for z}
	We assume that the prime $2\in \QQ$ is completely splitting in $F$.
		Suppose that $l \in 2\NN^{\Sigma_\infty}$ satisfies $\underline{l} := \min_{v \in \Sigma_\infty}l_v \ge 6$.
		Let $\gq$ vary in the set of non-zero prime ideals of $\go$.
	For $r \in \NN$, we assume $\Nice(f,r)$ for all $\gq$ and all $\pi\in \Pi_{\rm cus}^*(l, \gq)$.
	 For any $z \in [0,1]$, define $\beta_2>0$ by
	\[\beta_2 = \frac{1}{r(r\frac{\underline{l}-3-z+2d_F}{2d_F}+\frac{1}{2})}
	\times \begin{cases}
		\frac{1}{2} & (\underline{l}\ge d_F+4), \\
		\frac{\underline{l}-3-z}{2d_F} & (6\le \underline{l}\le d_F+3).
	\end{cases}\]
	Then, for any even Schwartz function $\phi$ on $\RR$
	with ${\rm supp}(\hat\phi) \subset (-\beta_2, \beta_2)$,
we have
	\begin{align*}&\lim_{\nr(\gq) \rightarrow \infty}
	\frac{1}{\sum_{\pi \in \Pi_{\rm cus}^*(l,\gq)}\frac{L(\frac{z+1}{2},\Sym^2(\pi))}{L(1,\Sym^2(\pi))}}\sum_{\pi \in \Pi_{\rm cus}^*(l,\gq)}\frac{L(\frac{z+1}{2}, \Sym^2(\pi))}{L(1, \Sym^2(\pi))}D(\Sym^r(\pi), \phi) \\
	= &
	\begin{cases} \hat\phi(0)-\frac{1}{2}\phi(0)=\int_{\RR} \phi(x)W({\rm Sp})(x)dx & (\text{$r$ is even and $(r,z)\neq (2,0)$}), \\
	\hat\phi(0)+\frac{1}{2}\phi(0)=\int_{\RR} \phi(x)W({\rm O})(x)dx & (\text{$r$ is odd}),\\
	\hat\phi(0)-\frac{3}{2}\phi(0) +2\int_{\RR}\hat\phi(x)|x|dx & (\text{$r=2$ and $z=0$}).
	\end{cases}
	\end{align*}
\end{thm}
Here we note that the denominator in the left-hand side above is non-zero as before.
Compared to Theorem \ref{RR-AB}, our Theorem \ref{main: weighted one level density for z} can be interpreted as a $z$-interpolation of Ricotta and Royer's result.
We remark that our assumption $\underline{l} \ge 6$ can be weakened to $\underline{l}\ge 4$, in which the condition $z \in [0,1]$ is replaced with $z\in [0, \min(1,\sigma)]$ for any $\sigma \in (0, \underline{l}-3)$.
This condition on $z$ and $l$ is derived from
the assumption in \cite[Corollary 1.2]{SugiyamaTsuzuki2018}.

There are some remarks in various directions.
Theorem \ref{main: weighted one level density for z} for $z=1$
is a generalization of Theorem \ref{RR-AB} to the case of Hilbert modular forms without harmonic weight.
Although the setting is slightly different, Theorem \ref{main: weighted one level density for z} for $z=1$ is similar to Shin and Templier's result \cite[Example 9.13 and Theorem 11.5]{ShinTemplier}, where they considered the principal congruence subgroup $\Gamma(\gq)$ instead of our level $\Gamma_0(\gq)$, the Hecke congruence subgroup.

Theorem \ref{main: weighted one level density for z} in the case of the standard $L$-functions ($r=1$) for Hilbert modular forms
without weight factor ($z=1$) has overlap with
Liu and Miller \cite{LiuMiller}, in which they assumed GRH and imposed conditions that the narrow class number of $F$ is one, the weight is the parallel weight $2k\ge 4$, and the level is endowed by the congruence subgroup $\Gamma_0(\Ical)$ for a square-free ideal $\Ical$ of $\go$. Their contribution is to extend the size of the support of $\hat\phi$ beyond $(-1,1)$ under the assumption above.

We remark that, when $r$ is odd,
Theorem \ref{main: weighted one level density for z} does not distinguish
$W({\rm O})$, $W(\SO({\rm even}))$ and $W(\SO({\rm odd}))$ because of $(-\beta_2,\beta_2) \subset [-1, 1]$.
Nevertheless we describe the assertion for odd $r$ by $W({\rm O})$
due to Ricotta and Royer's work \cite{RicottaRoyer}.
For determining the symmetry type, we need to calculate two-level density 
of low-lying zeros weighted by symmetric square $L$-functions.
It can be done by generalizing the trace formula in Theorem \ref{JZtrace} (\cite[Corollary 1.2]{SugiyamaTsuzuki2018})
to the case where $S$ and $S(\gn)$ used there have common finite places.

We summarize the change of densities in Theorem \ref{main: weighted one level density for z} when moving $z$ in $[0,1]$.
For any $r\neq 2$, the weighted density ($W(\Sp)$ or $W(\rm O)$) for the family of $L(s, \Sym^r(\pi))$ for $\pi \in \Pi_{\rm cus}^*(l, \gq)$ is stationary as $z$ varies in $[0,1]$.
Contrary to this, the weighted density for $r=2$ is stationary as $z \in (0,1]$
but the symmetry type $W(\rm Sp)$ is broken when $z=0$.
Hence, we conclude that the change of density of low-lying zeros of $L(s, \Sym^r(\pi))$ occurs only when
$r=2$ and the weight factors are essentially central values $L(1/2, \Sym^2(\pi))$.
Our weighted density for $(r, z)= (2, 0)$ is symplectic type plus a term,
which does not come from random matrix theory.
A density function not arising in random matrix theory is also seen for families of $L$-functions attached to elliptic curves in Miller \cite{Miller}.
Therefore, our Theorem \ref{main: weighted one level density for z} gives us a new example of the phenomenon that central $L$-values effect to change of density of low-lying zeros, as seen in Knightly and Reno \cite{KnightlyReno} for $\GL_2$ and Kowalski, Saha and Tsimerman \cite{KowalskiST}, Dickson \cite{Dickson} for ${\rm GSp}_4$.
From these observations, it might be meaningful to suggest the following,
which is not a rigorous form.
\begin{conj}
	Let $\Fcal=\bigcup_{k=1}^\infty \Fcal_k$ be a family of indices of $L$-functions
	{\rm(}e.g., a multiset of irreducible automorphic representations
	of an adelic group such as a harmonic family in the sense of \cite{SarnakShinTemplier}{\rm)}.
	Assume $\#\Fcal_k<\infty$ for each $k\ge1$ and $\lim_{k\rightarrow \infty}\#\Fcal_k=\infty$.
	Further assume existence of density $W(\Fcal)$ for one-level density
	of low-lying zeros of the family of $L$-functions $L(s, \Pi)$, $(\Pi \in \Fcal)$,
	that is,
	\[ \lim_{k\rightarrow \infty}\frac{1}{\#\Fcal_k}
	\sum_{\Pi \in \Fcal_k}D(\Pi,\phi)=\int_{\RR}\phi(x) W(\Fcal)(x)dx \]
	for Paley-Wiener functions $\phi$, where $D(\Pi,\phi)$ is the one-level density such as \eqref{def one-level density}.
	Let $w : \Fcal \rightarrow \CC$ be a function such that $\sum_{\Pi \in \Fcal_k}w_\Pi \neq 0$ for any $k\gg 1$
and	the weighted one-level density is described as
	\[\lim_{k\rightarrow \infty}\frac{1}{\sum_{\Pi \in \Fcal_k}w_\Pi }\sum_{\Pi \in \Fcal_k}w_\Pi\, D(\Pi,\phi) = \int_\RR\phi(x)W_w(\Fcal)(x)dx.\]
	Then, the density $W_w(\Fcal)$ would be changed from $W(\Fcal)$ only when $w_\Pi$ essentially contains the central value $L(1/2, \Pi)$.
\end{conj}
It might be better to consider the case where the weight factor $w_\Pi$ is a quotient of
special values of automorphic $L$-functions or a period integrals as long as $\Fcal$ consists of automorphic representations.
To attack the conjecture, relative trace formulas would be more useful tools
rather than the Arthur-Selberg trace formula.

\subsection{Framework}

The usual distribution of low-lying zeros has been mainly analyzed by usage of
the Petersson trace formula and the Kuznetsov trace formula, which are the same in the view point of relative trace formulas given by double integrals along the product of adelic quotients of two unipotent subgroups.
Shin and Templier used Arthur's invariant trace formula by
 estimating orbital integrals to quantify
automorphic Plancherel density theorem.
Contrary to those trace formulas, the proof of Theorem \ref{main: weighted one level density for z} relies on the explicit Jacquet-Zagier type trace formula
given by Tsuzuki and the author \cite{SugiyamaTsuzuki2018}, which is a generalization of the Eichler-Selberg trace formula to a parameterized version.
In our setting,
we need to treat not only the main term but also the second main term of the weighted automorphic Plancherel density
theorem quantitatively as in Theorem \ref{quantitative ver of equidist} in order to study weighted density of low-lying zeros.

This article is organized as follows.
In \S \ref{Refinements of equidistributions weighted by symmetric square L-functions},
we review the explicit Jacquet-Zagier trace formula for $\GL_2$ given by Tsuzuki and the author
\cite{SugiyamaTsuzuki2018}, and prove a refinement of the Plancherel density theorem
weighted by symmetric square $L$-functions \cite[Theorem 1.3]{SugiyamaTsuzuki2018}
by explicating the error term.
In \S \ref{Weighted distributions of low-lying zeros}, we introduce
symmetric $r$th $L$-functions for $\GL_2$ for any $r\in \NN$
and prove the main theorem \ref{main: weighted one level density for z}
on weighted density of low-lying zeros of such $L$-functions.
The latter part of this article from Appendix \ref{Comparison of the ST trace formula with Zagier's formula} to Appendix \ref{Comparison of the ST trace formula with Mizumoto's formula} is independent of our main results and is devoted to the comparison of
Tsuzuki and the author's trace formula \cite{SugiyamaTsuzuki2018} (called {\it the ST trace formula} in this article) 
with
Zagier's formula \cite{Zagier}, Takase's formula \cite{Takase} and Mizumoto's formula \cite{Mizumoto}.
This comparison is not so straightforward and non-trivial as the anonymous referee of \cite{SugiyamaTsuzuki2018} pointed out to the author.
Hence we prove that the ST trace formula recovers
Zagier's, Takase's and Mizumoto's formulas under conditions
that yield overlap with the ST trace formula, in 
Appendices \ref{Comparison of the ST trace formula with Zagier's formula}, \ref{Comparison of the ST trace formula with Takase's formula} and \ref{Comparison of the ST trace formula with Mizumoto's formula}, respectively.

\subsection{Notation}
Let $\NN$ be the set of positive integers
and set $\NN_0 = \NN\cup\{0\}$.
For a condition $\rm P$, $\delta({\rm P})$ is the generalized Kronecker delta symbol
defined by $\delta({\rm P})=1$ if $\rm P$ is true, and $\delta({\rm P})=0$ if $\rm P$ is false, respectively.
Throughout this article, any fractional ideal of $F$ is always supposed to be non-zero.
We set $\Gamma_\RR(s)=\pi^{-s/2}\Gamma(s/2)$ and $\Gamma_\CC(s)=2(2\pi)^{-s}\Gamma(s)$.

\section{Refinements of equidistributions weighted by symmetric square $L$-functions}
\label{Refinements of equidistributions weighted by symmetric square L-functions}
In this section, we give a quantitative version of the equidistribution
of Hecke eigenvalues weighted by $L(\frac{z+1}{2}, \Sym^2(\pi))$ in \cite[Theorem 1.3]{SugiyamaTsuzuki2018} as a deduction from the explicit Jacquet-Zagier trace formula in \cite{SugiyamaTsuzuki2018} by estimating error terms more explicitly.
For this purpose, let us review the Jacquet-Zagier type trace formula.

Let $F$ be a totally real number field of finite degree $d_F=[F:\QQ]$.
We suppose that the prime $2 \in \QQ$ is completely splitting
in $F$.
Let $D_F$ be the absolute value of the discriminant of $F/\QQ$.
The set of archimedean places (resp.\ non-archimedean places) of $F$ is denoted by $\Sigma_\infty$ (resp.\ $\Sigma_\fin$).
For an ideal $\ga$ of $\go$, we denote by $S(\ga)$ the set of all finite places dividing $\ga$
and by $\nr(\ga)$ the absolute norm $\nr(\ga)$, respectively.
For any $v \in \Sigma_\infty\cup\Sigma_\fin$, let $F_v$ be the completion of $F$ at $v$.
The normalized valuation of $F_v$ is denoted by $|\cdot|_v$.
For any $v \in \Sigma_\fin$, let $q_v$ be the cardinality of the residue field $\go_v/\gp_v$ of $F_v$, where $\go_v$ and $\gp_v$ are the integer ring of $F_v$ and its unique maximal ideal, respectively.
We fix a uniformizer $\varpi_v$ at every $v\in \Sigma_\fin$. Then,
$q_v=|\varpi_v|_v^{-1}$ holds.

For an ideal $\ga$ of $\go$,
the symbol $\ga=\square$ means that $\ga$ is a square of a non-zero ideal of $\go$.

\subsection{Trace formulas}
\label{Trace formulas}
We review the Jacquet-Zagier type trace formula given in \cite{SugiyamaTsuzuki2018}.
Let $S$ be a finite subset of $\Sigma_\fin$.
Let $\gn$ be a square-free ideal of $\go$ relatively prime to $S$ and $2\go$.
Let $\Pi_{\rm cus}(l,\gn)$ denote the set of all irreducible cuspidal
automorphic representations
$\pi\cong \otimes_v'\pi_v$ of $\PGL_2(\AA)$ such that $\pi_v$ for each $v \in \Sigma_\infty$ is isomorphic to the discrete series representation $D_{l_v}$
of $\PGL_2(\RR)$ whose minimal ${\rm O}(2)$-type is $l_v$ and the conductor $\gf_\pi$
of $\pi$ divides $\gn$.
For any $\pi \in \Pi_{\rm cus}^*(l,\gn)$ and $v \in \Sigma_\fin-S(\gf_\pi)$,
the Satake parameter of $\pi_v$ is denoted by $(q_v^{-\nu_v(\pi)/2}, q_v^{\nu_v(\pi)/2})$.
Notice Blasius' bound $|q_v^{\nu_v(\pi)/2}|=1$ by \cite{Blasius}.
For a complex number $z$, set
\[W_\gn^{(z)}(\pi):= \nr(\gn\gf_\pi^{-1})^{(1-z)/2}\prod_{v \in S(\gn\gf_\pi^{-1})}
\left\{ 1+\frac{Q(I_v(|\cdot|_v^{z/2}))-Q(\pi_v)^2}{1-Q(\pi_v)^2} \right\},\]
where $I_v(\chi_v) = {\rm Ind}_{B(F_v)}^{\GL_2(F_v)}(\chi_v\boxtimes\chi_v^{-1})$
for a quasi-character $\chi_v : F_v^\times \rightarrow \CC^\times$
denotes the normalized parabolic induction, and we set
\begin{align*}
Q(\tau) = \frac{a_v+a_v^{-1}}{q_v^{1/2}+q_v^{-1/2}}
\end{align*}
for spherical representations $\tau$ of $\PGL_2(F_v)$ with the Satake parameter $(a_v, a_v^{-1})$.

Let $\Acal_v$ for each $v \in \Sigma_{\fin}$ be the space consisting of all
holomorphic functions $\a_v$ on $\CC/(4\pi i(\log q_v)^{-1}\ZZ)$ such that $\a_v(-s_v)=\a_v(s_v)$, and set $\Acal_S=\bigotimes_{v \in S} \Acal_v$.
Then, for any $\a \in \Acal_S$,
we define the quantity $\II_{\rm cusp}^0(\gn|\a,z)$ arising in the cuspidal part of the spectral side in the explicit Jacquet-Zagier type trace formula as
\[\II_{\rm cusp}^0(\gn|\a,z) = \frac{1}{2}D_F^{z-1/2}\nr(\gn)^{(z-1)/2}
\sum_{\pi \in \Pi_{\rm cus}(l,\gn)} W_{\gn}^{(z)}(\pi) \frac{L(\frac{z+1}{2}, \Sym^2(\pi))}{L(1,\Sym^2(\pi))}\a(\nu_S(\pi))\]
with $\nu_S(\pi) = (\nu_v(\pi))_{v \in S}$.
Here $L(s, \Sym^2(\pi))$ is the completed symmetric square $L$-function associated with $\pi$.
The following is also needed to describe the spectral side:
\[C(l,\gn) := D_F^{-1}\prod_{v \in \Sigma_\infty}\frac{4\pi}{l_v-1}\prod_{v \in S(\gn)}\frac{1}{1+q_v}.\]

Next we define the quantities arising in the geometric side of the explicit Jacquet-Zagier type trace formula.
For $v \in \Sigma_\infty$, we set
\begin{align*}
\Ocal_v^{+,(z)}(a)=\frac{2\pi}{\Gamma(l_v)}
\frac{\Gamma(l_v+\frac{z-1}{2})\Gamma(l_v+\frac{-z+1}{2})}
{\Gamma_\RR(\frac{1+z}{2})\Gamma_\RR(\frac{1-z}{2})}
\delta(|a|>1)(a^2-1)^{1/2}\mathfrak{P}_{\frac{z-1}{2}}^{1-l_v}(|a|)
\end{align*}
and
\begin{align*}
\Ocal_v^{-,(z)}(a)=\frac{\pi}{\Gamma(l_v)}
\Gamma\left(l_v+\tfrac{z-1}{2}\right)\Gamma\left(l_v+\tfrac{-z+1}{2}\right)
\sgn(a)(1+a^2)^{1/2}\{\mathfrak{P}_{\frac{z-1}{2}}^{1-l_v}(ia)
-\mathfrak{P}_{\frac{z-1}{2}}^{1-l_v}(-ia)\}
\end{align*}
for $a \in F_v\cong \RR$,
where $\fP_{\nu}^{\mu}(x)$ is the associated Legendre function of the first kind
defined on $\CC- (-\infty, 1]$ (cf.\ \cite[\S 4.1]{MOS}). 

For $v \in \Sigma_\fin$,
let $\varepsilon_\delta$ for $\delta \in F_v^\times$ denote the real-valued character of
$F_v^\times$ corresponding to $F_v(\sqrt{\delta})/F_v$ by local class field theory.
We define functions on $F_v^\times$ associated with $\delta \in F_v^\times$, $\e\in\{0,1\}$ and $z,s \in \CC$ satisfying
$\Re(s)>(|\Re(z)|-1)/2$ by
{\allowdisplaybreaks
	\begin{align*}
	\Ocal_{v,\e}^{\delta,(z)}(a)&=\frac{\zeta_{F_v}(-z)}{L_{F_v}\left(\tfrac{-z+1}{2},\varepsilon_{\delta}\right)}\biggl(\frac{1+q_v^{\frac{z+1}{2}}}{1+q_v}\biggr)^{\e}
	|a|_v^{\frac{-z+1}{4}}
	+\frac{\zeta_{F_v}(z)}{L_{F_v}\left(\tfrac{z+1}{2},\varepsilon_{\delta}\right)}
	\biggl(\frac{1+q_v^{\frac{-z+1}{2}}}{1+q_v}\biggr)^{\e}
	|a|_v^{\frac{z+1}{4}}
	\end{align*}
}for all $a \in F_v^\times$, and
{\allowdisplaybreaks\begin{align*}
	\Scal_{v}^{\delta,(z)}(s;a)
	&=-q_v^{-\frac{s+1}{2}}\frac{\zeta_{F_v}\left(s+\frac{z+1}{2}\right)\zeta_{F_v}\left(s+\tfrac{-z+1}{2}\right)}{L_{F_v}(s+1,\varepsilon_{\delta})}
	|a|_v^{\frac{s+1}{2}}, \quad (|a|_v\leq 1), 
	\\
	\Scal_v^{\delta,(z)}(s;a)
	&=-q_v^{-\frac{s+1}{2}}
	\left\{\frac{\zeta_{F_v}(-z)\zeta_{F_v}\left(s+\tfrac{z+1}{2}\right)}{L_{F_v}\left(\tfrac{-z+1}{2},\varepsilon_{\delta}\right)}\left|a\right|_v^{\frac{-z+1}{4}}
	+\frac{\zeta_{F_v}(z)\zeta_{F_v}\left(s+\tfrac{-z+1}{2}\right)}{L_{F_v}\left(\tfrac{z+1}{2},\varepsilon_{\delta}\right)}\left|a\right|_v^{\frac{z+1}{4}}\right\}, \quad (|a|_v>1).
	\notag
	\end{align*}
}Here $\zeta_{F_v}(s)$ and $L_{F_v}(s, \varepsilon_\delta)$ denote the local $L$-factors
attached to the trivial character of $F_v^\times$ and to $\varepsilon_\delta$, respectively.
Furthermore, for a function $\a_v \in \Acal_v$ and $a\in F_v^\times$, set
\[
\hat \Scal_v^{\delta,(z)}(\a_v;a) =
\tfrac{1}{2\pi i}\int_{c - 2 \pi i(\log q_v)^{-1}}^{c + 2\pi i(\log q_v)^{-1}} \Scal_v^{\delta,(z)}(s;a)\,\a_v(s)\,\frac{\log q_v}{2}(q_v^{(1+s)/2}-q_v^{(1-s)/2})ds
\]
for some $c \in \RR$.

Take a function $\a \in \Acal_S$. If $\a$ is a pure tensor of the form $\otimes_{v\in S} \a_v$, we set
\[
{\bf B}_{\fn}^{(z)}(\a|\Delta;\fa)=\prod_{v\in \Sigma_\fin-S\cup S(\fn)}\Ocal_{0,v}^{\Delta,(z)}(a_v)\prod_{v\in S(\fn)}\Ocal_{1,v}^{\Delta,(z)}(a_v)\prod_{v\in S}
\hat\Scal_v^{\Delta,(z)}(\a_v,a_v)
\]
for any ideal $\ga$ of $\go$, any $\Delta \in F^\times$ and any $z \in \CC$.
Further we set
\[
\Upsilon_v^{(z)}(\a_v)=
\prod_{v \in S}
\tfrac{1}{2\pi i} \int_{c_v-2\pi i (\log q_v)^{-1}}^{c_v+2\pi i(\log q_v)^{-1}}
\frac{-q_v^{-{(s_v+1)}/{2}}}{1-q_v^{-s_v-(z+1)/{2}}}\a_v(s_v)\frac{\log q_v}{2}
(q_v^{(1+s_v)/2}-q_v^{(1-s_v)/2})ds_v\]
with a fixed sufficiently large $c_v \in \RR$ for each $v \in S$
and set $\Upsilon^{(z)}(\a)=\prod_{v \in S}\Upsilon_v^{(z)}(\a_v)$.

From functions defined above, the quantities in the geometric side are defined as $\JJ_{\rm unip}^{0}(\gn|\a,z)$, $\JJ_{\rm hyp}^{0}(\gn|\a,z)$ and $\JJ_{\rm ell}^{0}(\gn|\a,z)$.
First set
\[\JJ_{\rm unip}^{0}(\gn|\a,z) = D_F^{-\frac{z+2}{4}}\zeta_F(-z)\Upsilon^{(z)}(\a)\prod_{v \in S(\gn)}\frac{1+q_v^{\frac{z+1}{2}}}{1+q_v}\prod_{v \in \Sigma_{\infty}}2^{1-z}\pi^{\frac{3-z}{4}}\frac{\Gamma(l_v+\frac{z-1}{2})}{\Gamma(\frac{z+1}{4})\Gamma(l_v)}\]
and
\[\JJ_{\rm hyp}^{0}(\gn|\a,z) =
\tfrac{1}{2}D_F^{-1/2}\zeta_F(\tfrac{1-z}{2})
\sum_{a \in \go(S)^\times_{+}-\{1\}}
\bfB_\gn^{(z)}(\a|1;a(a-1)^{-2}\go)
\prod_{v \in \Sigma_\infty}\Ocal_v^{+,(z)}(\tfrac{a+1}{a-1}),
\]
where $\zeta_F(s)$ is the completed Dedekind zeta function of $F$
and $\go(S)_{+}^{\times}$ is the totally positive unit group of the $S$-integers of $F$.
For any $\Delta \in F^\times$ such that $\sqrt{\Delta} \notin F^\times$,
let $\varepsilon_\Delta$ be the real-valued character of $F^\times \bsl \AA^\times$
corresponding to $F(\sqrt{\Delta})/F$ by class field theory, and
$L(s,\varepsilon_\Delta)$ the completed Hecke $L$-function associated with $\varepsilon_\Delta$.

Then, we set
\[\JJ_{\rm ell}^{0}(\gn|\a,z)=\frac{1}{2}D_F^{\frac{z-1}{2}}\sum_{(t:n)_F}
\nr(\mathfrak{d}_\Delta)^{\frac{z+1}{4}}L(\tfrac{z+1}{2},\varepsilon_\Delta)
\bfB_\gn^{(z)}(\a|\Delta;n\gf_\Delta^{-2})
\prod_{v \in \Sigma_{\infty}}\Ocal_v^{\sgn(\Delta^{(v)}),(z)}(t|\Delta|_v^{-1/2}),
\]
where $\mathfrak{d}_\Delta$ is the relative discriminant of $F(\sqrt{\Delta})/F$
and $\gf_\Delta$ is the fractional ideal of $F$ satisfying $\Delta\go=\mathfrak{d}_\Delta\gf_\Delta^2$.
In the summation defining $\JJ_{\rm ell}^0(\gn|\a,z)$,
$(t:n)_F$ runs over the different cosets $\{(ct, c^2 n) \in F \times F \ | \ c \in F^\times \}$
such that $\Delta=t^2-4n \in F^\times -(F^\times)^2$,
$(t,n) \in \{(c_v t_v , c_v^2 n_v) \ | \ c_v\in F_v^\times, t_v \in \go_v, n_v \in \go_v^\times \}$
for all $v \in \Sigma_\fin-S$, and $\ord_v(n\gf_\Delta^{-2})<0$ for all $v \in S(\gn)$
with $\varepsilon_{\Delta,v}$ being unramified and non-trivial.

From the preparation so far, the explicit Jacquet-Zagier type trace formula is given
as follows.
\begin{thm}\label{JZtrace}
	{\rm(}\cite[Corollary 1.2]{SugiyamaTsuzuki2018}{\rm)}
Let $S$ be a finite set of $\Sigma_{\fin}$.
	Let $l=(l_v)_{v \in \Sigma_\infty}$ be a family of positive even integers such that $\min_{v \in \Sigma_{\infty}}l_v\ge 4$, and $\gn$ a square-free ideal of $\go$ relatively prime to $2\go$ and $S$.
For any $z \in \CC$ with $|\Re(z)|<\min_{v \in \Sigma_{\infty}}l_v-3$, we have
\[(-1)^{\#S}C(l,\gn)\II_{\rm cusp}^0(\gn|\a,z) = D_F^{z/4}\{\JJ_{\rm unip}^0(\gn|\a,z)
+\JJ_{\rm unip}^0(\gn|\a,-z) \} +\JJ_{\rm hyp}^0(\gn|\a,z) +\JJ_{\rm ell}^0(\gn|\a,z).\]
\end{thm}
\begin{rem}
{\rm In \cite{SugiyamaTsuzuki2018}, It is assumed that $S$ does not include the dyadic places.
This assumption is removable by \cite[\S5]{SugiyamaTsuzuki2019}.}
\end{rem}

\subsection{Quantitative versions of weighted equidistributions}
In this subsection, we give a quantitative version of weighted equidistribution of the Satake parameters of $\pi \in \Pi_{\rm cus}(l,\gn)$.
We give estimates of geometric terms $\JJ_{\rm unip}^0(\gn|\a,z)$, $\JJ_{\rm hyp}^0(\gn|\a,z)$
and $\JJ_{\rm ell}^0(\gn|\a,z)$ by making the dependence on the test function $\a$ and $S$ explicit.
In what follows, we do not mention the dependence on $F$ and $l$ of the implied constants.

Let $S$ be a finite subset of $\Sigma_{\fin}$.
For $v \in S$, let $\Acal_v^0$ be the space of all Laurent polynomials in $q_v^{-s_v/2}$ which is invariant under
$q_v^{-s_v/2} \mapsto q_v^{s_v/2}$.
Then $\Acal_v^0$ has a $\CC$-basis $(q_v^{-s_v/2})^{n} + (q_v^{s_v/2})^{n}$ for all $n \in \NN_0$.
Let $\Acal_v^0[m]$ for $m \in \NN_0$ be the subspace of $\Acal_v^0$ generated by
$(q_v^{-s_v/2})^{n}+ (q_v^{s_v/2})^{n}$ for all $n$ with $0\le n \le m$.

Let $\ga$ be an integral ideal of the form $\ga=\prod_{v \in S} (\gp_v\cap \go)^{n_v}$ with $n_v \in \NN$. Set $\Acal(\ga) =\otimes_{v \in S}\Acal_v^0[n_v]$.
We fix a large $c>1$ and define $\bfc=(c_v)_{v \in S} \in \RR^S$ by $c_v=c$ for all $v \in S$.
For $\a \in \Acal_S$, set $\|\a\|=(\frac{1}{2\pi})^{\#S}\int_{\LL_S(\bfc)}|\a(\bfs)|\,
|d\mu_S(\bfs)|$, which depends on $\bfc$.
Let $X_{n}$ for $n \in \NN_0$ be the polynomial defined by
$\frac{\sin (n+1)\theta}{\sin \theta} = X_n(2\cos\theta)$.
The following lemma is clear by $\sup_{x \in [-2,2]}|X_n(x)|\ll n+1$.
\begin{lem}\label{esti of alpha}
For any $\a(\bfs) =\prod_{v \in S}X_{m_v}(q_v^{-s/2}+q_v^{s/2}) \in \Acal(\ga)$, we have
\[\|\a\| \ll_{\bfc} \prod_{v \in S}(m_v+1)\ll_\e\nr(\ga)^\e\]
for any $\e>0$ where the implied constant is independent of $\a$ and $\ga$.
\end{lem}

\begin{lem}\label{hyperbolic term}
Fix any $\s \in [0,1]$ and suppose $\underline{l}:=\min_{v \in \Sigma_\infty}l_v > \s+3$.
For any $\a \in \Acal(\ga)$ and $z \in \CC$ with $|\Re(z)|\le \s$, we have
\[\JJ_{\rm hyp}^0(\gn|\a,z) \ll_{\s,\e,\e'} \nr(\ga)^{(\frac{\underline{l}}{2}+d_F-1-(\s+1)/2-\e)/d_F+\e'}
\|\a\| \, \nr(\gn)^{-\delta+\e}\]
uniformly in $z$, $\a$, $\gn$ and $\ga$ for any small $\e, \e'>0$, where $\delta \in (1/2,1]$ is defined by
\[\delta:=\begin{cases}
1 & (\s\le \underline{l}-3-d_F), \\
\frac{1}{2}+\frac{\underline{l}-3-\s}{2d_F} & (\s>\underline{l}-3-d_F).
\end{cases}\]
\end{lem}
\begin{proof}
This is proved by explicating the implied constant depending on $\a$ and $\ga$ in \cite[Lemma 8.2]{SugiyamaTsuzuki2018}.
By the proof of \cite[Lemma 8.2]{SugiyamaTsuzuki2018}, we have the estimate
\[\JJ_{\rm hyp}^0(\gn|\a,z)\ll_{\s,\e}C^{\#S}\|\a\|
\sum_{\gc|\gn}
\{\prod_{v \in S(\gn\gc^{-1})}\tfrac{4}{q_v+1} \prod_{v \in S(\gc)}\tfrac{4(q_v^{1/2}+1)}{q_v+1}\}\sum_{x \in \gc\ga^{-1}-\{0\}}f_\infty(x+1),\]
where $f_\infty((a_v)_{v\in\Sigma_{\infty}}) =\prod_{v \in \Sigma_{\infty}}f_v(a_v)$ with
$f_v(a_v)=\delta(a_v>0)(1+|a_v|_v)^{-\frac{l_v}{2}+\frac{\s+1}{2}+\e}$, and $C>0$ is an absolute constant.
By
$\sum_{x \in \gc\ga^{-1}-\{0\}}f_\infty(x+1)\ll \nr(\ga)\nr(\gc\ga^{-1})^{\{1-\underline{l}/2+(\s+1)/2+\e\}/d_F}$ uniformly for $\gc$ and $\ga$ from \cite[Lemma 7.19]{SugiyamaTsuzuki2018}, we conclude the assertion.
\end{proof}

\begin{lem}\label{elliptic term}
For any $\a \in \Acal(\ga)$, we have
\[\JJ_{\rm ell}^0(\gn|\a,z)\ll_{\e}\nr(\ga)^{\e} \|\a\| \, \nr(\gn)^{-\delta'+\e}
\]
uniformly in $z\in \CC$ with $|\Re(z)| \le 1$, $\a$, $\gn$ and $\ga$ for any small $\e>0$, where $\delta' \in (1/2,1]$ is defined by
\[\delta':=\begin{cases}
1 & (\underline{l}>d_F+2), \\
\frac{1}{2}+\frac{\underline{l}-2}{2d_F} & (d_F+2 \ge \underline{l}).
\end{cases}\]
\end{lem}
\begin{proof}
This is proved by explicating the implied constant depending on $\a$ and $\ga$ in \cite[Lemma 8.3]{SugiyamaTsuzuki2018}.
In the proof, the assumption \cite[(7.11)]{SugiyamaTsuzuki2018} is not imposed although we need it as in \cite[\S7]{SugiyamaTsuzuki2018}.
This is because we can take a complete system $\{\ga_j\}_{j=1}^{h}$ of representatives for the ideal class group of $F$ such that $\{\ga_{j}\}_{j=1}^{h}$ are prime ideals relatively prime to $S$ satisfying \cite[(7.11)]{SugiyamaTsuzuki2018} as long as we fix $\gn$.

By the proof of \cite[Lemma 8.3]{SugiyamaTsuzuki2018}
with the aid of \cite[Lemma 8.1]{SugiyamaTsuzuki2018}, we have
\[\JJ_{\rm ell}^0(\gn|\a,z)\ll C^{\#S}\|\a\|\sum_{(\gc,\gn_1,i,\varepsilon,\nu)}\Xi^{(z,\bfc)}(\gc,\gn_1,i,\varepsilon n_{i,\nu})\]
with an absolute constant $C>0$, where $(\gc,\gn_1,i,\varepsilon,\nu)$ varies so that
$\gn_1$ and $\gc$ are integral ideals such that $\gn_1|\gn$ and $\gc|\gn_1$,
$\varepsilon \in \go^\times/(\go^\times)^2$, $\nu\in\NN_0^S$ with $\nu_v \le \ord_v(\ga)$ for all $v \in S$, and
$1\le i \le h$ with $n_{i,\nu}\go = \ga_i\prod_{v \in S}\gp^{\nu_v}$.
To estimate $\Xi^{(z,\bfc)}(\gc,\gn_1,i,\varepsilon n_{i,\nu})$ as in \cite[Lemma 7.20]{SugiyamaTsuzuki2018},
the dependence on $\ga$ and $S=S(\ga)$ occurs
in the arguments in \cite[p.3028]{SugiyamaTsuzuki2018}\footnote
{In \cite[p.3029]{SugiyamaTsuzuki2018}, the factor $\prod_{v \in S}|\Delta_v^0|_v\times\prod_{v \in \Sigma_{\fin}-S}|4^{-1}|_v$ depending on $S$ occurs. This is negligible since it is estimated as
$\prod_{v \in S}|\Delta_v^0|_v\times\prod_{v \in \Sigma_{\infty}\cup S}|4|_v \le 4^{d_F}$.
}
(cf.\ \cite[Lemma 7.14 (3)]{SugiyamaTsuzuki2018}),
in \cite[p.3032]{SugiyamaTsuzuki2018}, and in the second term of the majorant in \cite[Lemma 7.20]{SugiyamaTsuzuki2018}.
The implied constant depending on $\ga, S$ in the first case is $16^{\#S}\ll_{\e}\nr(\ga)^{\e}$.
The constant in the second case is
\[\sum_{\substack{\nu=(\nu_v)_{v \in S} \\
 0\le \nu_v\le\ord_v(\ga) (\forall v \in S)}}
\hspace{-2mm}\{\prod_{v \in S}q_v^{\nu_v(\frac{-c_v+\varrho(z)}{4}+\e)}\}
|\nr(n_{i,\nu})|^{\frac{1}{2}+\frac{\underline{L}(z)-1}{2d_F}}
\ll (\sum_{\gc|\ga}1)\nr(\ga)^{\frac{1-c}{4}+\e+1/2 +\frac{\underline{l}-2-2\e}{2d_F} }
\ll_{\e} 1,
\]
where $\varrho(z)=\max(|\Re(z)|, 1)$ and $\underline{L}(z) = \underline{l}-\frac{1+\varrho(z)}{2}-2\e$ (cf.\ \cite[Lemma 7.20]{SugiyamaTsuzuki2018}).
Here we note $\varrho(z)=1$ and $\underline{L}(z)=\underline{l}-1-2\e$ by $|\Re(z)|\le1$,
and $c$ can be taken so that $\frac{1-c}{4}+\e+1/2 +\frac{\underline{l}-2-2\e}{2d_F} < 0$.
The constant in the third case is
\[\sum_{\nu \in \{0,1\}^S}\prod_{v \in S}\max(1,q_v^{\frac{-\Re(c_v)+|\Re(z)|}{4}+\e})
=2^{\# S}\ll\nr(\ga)^\e\]
by virtue of $|\Re(z)|\le 1$ and $c>1$.
From these, we obtain
\[\JJ_{\rm ell}^0(\gn|\a,z)\ll_{\e}\nr(\ga)^{\e} \|\a\|
\times\left\{\nr(\ga)^{\e}\sum_{j=1}^{h}\Sigma(\gn,N_j) +\nr(\ga)^{\e}\nr(\gn)^{-1+\e}\right\}\]
uniformly for $|\Re(z)|\le 1$,
where $\Sigma(\gn, N_j)$ is the series defined in \cite[p.3038]{SugiyamaTsuzuki2018}
and estimated 
in the same way as \cite[p.3038--3039]{SugiyamaTsuzuki2018}. Thus we are done.
\end{proof}

Let $\gn$ be a square-free ideal of $\go$ such that $\gn \neq \go$.
For any $f \in C([-2,2]^S)$,
we set
\[\Lambda_{\gn}^{(z)}(f)=\frac{1}{{\rm M}(\gn)^{\delta(z=0)}}\prod_{v \in S(\gn)}\frac{q_v^{(z-1)/2}}{1+q_v^{(z+1)/2}}\sum_{\pi \in \Pi_{\rm cus}(l,\gn)}
W_{\gn}^{(z)}(\pi)\frac{L(\frac{z+1}{2}, \Sym^2(\pi))}{L(1, \Sym^2(\pi))}f(\bfx_S(\pi)),
\]
where
\[{\rm M}(\gn)=\sum_{v \in S(\gn)}\frac{\log q_v}{1+q_v^{-1/2}},\qquad \bfx_S(\pi) = (q_v^{-\nu_v(\pi)/2}+q_v^{\nu_v(\pi)/2})_{v \in S}.\]
Let $\{\ga_j\}_{j=1}^{h}$ be a complete system of representatives
for the ideal class group of $F$ consisting of prime ideals relatively prime to $S$,
due to the Chebotarev density theorem.
We denote by $\zeta_{F, \fin}(z)$ the non-completed Dedekind zeta function of $F$
and ${\rm CT}_{z=1}\zeta_{F,\fin}(z)$ denotes the constant term of the Laurent expansion
of $\zeta_{F,\fin}(z)$ at $z=1$.

For any $\a\in\Acal$,
we denote by $f_\a$ the function $f_\a : [-2,2]^S \rightarrow \CC$ determined by $f_\a((q_v^{-s_v/2}+q_v^{s_v/2})_{v \in S})=\a(\bfs)$.
We refine the weighted equidistribution theorem \cite[Theorem 1.3 (1)]{SugiyamaTsuzuki2018} restricted to the function space $\Acal(\ga)$ quantitatively
by making the dependence on $\ga$ explicit as follows.

\begin{thm}\label{quantitative ver of equidist}
Let $l=(l_v)_{v \in \Sigma_\infty}$ be a family of positive even integers such that $\underline{l}:=\min_{v \in \Sigma_\infty}l_v\ge 4$.
Let $\ga$ be an ideal of the form $\prod_{v \in S}(\gp_v\cap\go)^{n_v}$ with $n_v \in \NN$. Take any $\a\in \Acal(\ga)$.
Let $\gn\neq \go$ be a square-free ideal of $\go$ relatively prime to $2\prod_{j=1}^h \ga_j$ and $S$.
For a fixed $\s \in \RR$ with $0 < \s < \underline{l}-3$,
take $z \in [0, \min(1,\s)]$.
Then, we have 
\begin{align}
& \Lambda_\gn^{(z)}(f_\a) \notag\\
=\, & \zeta_{F,\fin}(1+z)C_l^{(z)}(-1)^{\#S}\Upsilon^{(z)}(\a)
\bigg(1+D_F^{-3z/2}\prod_{v \in S(\gn)} \frac{1+q_v^{\frac{-z+1}{2}}}{1+q_v^{\frac{z+1}{2}}}\bigg)
\label{main of lambda z}\\
&+\|\a\|\Ocal_{\e,\e'}\left( \nr(\ga)^{(\underline{l}/2+d_F-1-(\s+1)/2-\e)/d_F+\e'}\nr(\gn)^{-\delta_1 +\e}
\right), \label{error of lambda z}
\end{align}
for any small $\e,\e'>0$ if $z>0$. Furthermore, we have
{\allowdisplaybreaks\begin{align}
& \Lambda_\gn^{(0)}(f_\a) \notag \\
= &
\frac{1}{{\rm M}(\gn)}
2\,\Res_{z=1}\zeta_{F,\fin}(z) \,C_l^{(0)}
\bigg[(-1)^{\#S}\Upsilon^{(0)}(\a)\bigg\{ d_F\left(-\frac{1}{2}\log 2-\frac{3}{4}\log \pi+\frac{1}{4}\psi\left(\frac{3}{4}\right)\right)
\label{main of lambda 01}\\
& +\frac{1}{2}\sum_{v \in \Sigma_{\infty}}
\psi(l_v-\tfrac{1}{2})\bigg\}
+(-1)^{\#S}\sum_{v \in S}\frac{d}{dz}\bigg|_{z=0}\Upsilon_v^{(z)}(\a)\prod_{\substack{w \in S \\ w\neq v}}\Upsilon_{w}^{(0)}(\a)
\bigg]\label{main of lambda 02}\\
&+\frac{1}{{\rm M}(\gn)}2\,{\rm CT}_{z=1}\zeta_{F, \fin}(z) \,C_l^{(0)}(-1)^{\#S}\Upsilon^{(0)}(\a)\label{main of lambda 03}\\
&+\Res_{z=1}\zeta_{F, \fin}(z)
\left(1 +\frac{3\log D_F}{2{\rm M}(\gn)} \right)C_l^{(0)}(-1)^{\#S}\Upsilon^{(0)}(\a) \label{main of lambda 04}\\
&+\|\a\|\,\Ocal_{\e,\e'}\left(  \nr(\ga)^{(\underline{l}/2+d_F-1-(\s+1)/2-\e)/d_F+\e'}
\frac{\nr(\gn)^{-\delta_1 +\e}}{{\rm M}(\gn)}\right) \label{error of lambda 0}
\end{align}
}for any small $\e,\e'>0$ if $z=0$.
Here the implied constants above are independent of $\gn$ and $\ga$,
and we set $\delta_1= \min(\delta,\delta')-1/2 \in (0,1/2]$, where $\delta$ and $\delta'$
are as in Lemmas \ref{hyperbolic term} and \ref{elliptic term}, respectively.
The function $\psi$ denotes the digamma function.
\end{thm}
\begin{proof}
 By the explicit Jacquet-Zagier type trace formula in Theorem \ref{JZtrace}, we have
{\allowdisplaybreaks\begin{align*}\Lambda_{\gn}^{(z)}(f_\a)
= & 2D_{F}^{1/2-z}\frac{1}{{\rm M}(\gn)^{\delta(z=0)}}\bigg(\prod_{v \in S(\gn)}\frac{1}{1+q_v^{\frac{z+1}{2}}}\bigg)
\II_{\rm cusp}^{0}(\gn| \a ,z) \\ 
= & \frac{2D_{F}^{1/2-z}}{{\rm M}(\gn)^{\delta(z=0)}}
\bigg(\prod_{v \in S(\gn)}\frac{1}{1+q_v^{\frac{z+1}{2}}}\bigg)(-1)^{\# S}C(l,\gn)^{-1}\\
&\times \bigg[ D_F^{z/4}\{\JJ_{\rm unip}^0(\gn|\a, z) +\JJ_{\rm unip}^0(\gn|\a, -z) \}
+\JJ_{\rm hyp}^{0}(\gn|\a, z) +\JJ_{\rm ell}^0(\gn | \a, z) \bigg].
\end{align*}
}Let us evaluate the unipotent terms $\JJ_{\rm unip}^{0}(\gn|\a,\pm z)$.
When $0<z<1$, we obtain
\begin{align*}
& \frac{2D_{F}^{1/2-z}}{{\rm M}(\gn)^{\delta(z=0)}}
\bigg(\prod_{v \in S(\gn)}\frac{1}{1+q_v^{\frac{z+1}{2}}}\bigg)(-1)^{\# S}C(l,\gn)^{-1} D_F^{z/4}\JJ_{\rm unip}^0(\gn|\a, z) \\
= & \zeta_{F,\fin}(1+z) C_l^{(z)} \, (-1)^{\#S}\Upsilon^{(z)}(\a)
\end{align*}
and
\begin{align*}
	& \frac{2D_{F}^{1/2-z}}{{\rm M}(\gn)^{\delta(z=0)}}\bigg(\prod_{v \in S(\gn)} \frac{1}{1+q_v^{\frac{z+1}{2}}}\bigg)(-1)^{\# S}C(l,\gn)^{-1} D_F^{z/4}\JJ_{\rm unip}^0(\gn|\a, -z) \\
	= & \bigg(\prod_{v \in S(\gn)} \frac{1+q_v^{\frac{-z+1}{2}}}{1+q_v^{\frac{z+1}{2}}}\bigg) 
	D_F^{-3z/2}
	\times \zeta_{F,\fin}(1-z)C_l^{(-z)} \times(-1)^{\#S}\Upsilon^{(-z)}(\a).
\end{align*}
Next consider the case $z=0$. By the expression
\[D_F^{-3z/2}\prod_{v \in S(\gn)}\frac{1+q_v^{\frac{-z+1}{2}}}{1+q_v^{\frac{z+1}{2}}}= 1+k_1z+\Ocal(z^2),\]
\[\zeta_{F,\fin}(1+z)=\frac{c_{-1}}{z}+c_0+\Ocal(z),\]
\[C_l^{(z)}(-1)^{\#S}\Upsilon^{(z)}(\a) = a_0+a_1z+\Ocal(z^2),\]
with
\[a_0=C_l^{(0)}(-1)^{\#S}\Upsilon^{(0)}(\a),\]
\begin{align*}a_1 = & \, C_l^{(0)} (-1)^{\#S}\Upsilon^{(0)}(\a) \times \bigg[ d_F\left\{-\frac{1}{2}\log 2-\frac{3}{4}\log \pi+\frac{1}{4}\psi\left(\frac{3}{4}\right)\right\} \\
& +\frac{1}{2}\sum_{v \in \Sigma_{\infty}}
\psi(l_v-\tfrac{1}{2})+\sum_{v \in S}\frac{\frac{d}{dz}|_{z=0}\Upsilon_v^{(z)}(\a)}{\Upsilon_v^{(z)}(\a)}
\bigg],
\end{align*}
\[k_1= -\frac{3}{2}\log D_F -{\rm M}(\gn),\]
we evaluate the unipotent terms as
\begin{align*}
&\bigg[2D_{F}^{1/2-z}\frac{1}{{\rm M}(\gn)^{\delta(z=0)}}
\bigg(\prod_{v \in S(\gn)}\frac{1}{1+q_v^{\frac{z+1}{2}}}\bigg)(-1)^{\# S}C(l,\gn)^{-1} \\
&\times D_F^{z/4}\{\JJ_{\rm unip}^0(\gn|\a, z)+\JJ_{\rm unip}^0(\gn|\a,-z)\}\bigg]\bigg|_{z=0}
= \, \frac{1}{{\rm M}(\gn)}\{2(c_{-1}a_1+c_0a_0)-k_1c_{-1}a_0\}.
\end{align*}
By combining these with Lemmas \ref{hyperbolic term} and \ref{elliptic term},
we are done.
\end{proof}
\begin{rem}{\rm 
	We need to correct \cite[Theorem 1.3]{SugiyamaTsuzuki2018} as follows.
	In the assertion, \\
	$\sup_{z \in [0,\min(1,\sigma)]}$ is considered.
	However, this should be replaced with $\sup_{z \in \{0\}\cup[\e, \min(1,\sigma)] }$
	with any fixed $\e \in (0,\min(1,\sigma))$,
	since $\JJ_{\rm unip}^{0}(\gn | \a,-z)$ for $z\neq 0$ yields the error term
	\\
	$\zeta_{F, \fin}(1-z)\nr(\gn)^{-z+\e}\Upsilon^{(-z)}(\a)$, which is not bounded in $z\in (0,1]$.}
\end{rem}

We note the identity
$(-1)^{\#S}\Upsilon^{(z)}(\a)= \langle \otimes_{v \in S}\l_v^{(z)}, f_\a \rangle$,
where $\otimes_{v \in S}\lambda_v^{(z)}$ is the measure on $[-2,2]^S$
defined in \cite[\S1.2.1]{SugiyamaTsuzuki2018}.
Indeed, we have the following.

\begin{lem}\label{explicit unip int}
	Let $z$ be a real number such that $z \in [-1,1]$.
For $m_v \in \NN_0$ and $\a_v(s) = X_{m_v}(q_v^{-s/2}+q_v^{s/2})$,
we have
\[-\Upsilon_v^{(z)}(\a_v) = \delta(m_v\in 2\NN_0)\,q_v^{-m_v(z+1)/4}.\]
\end{lem}
\begin{proof}
It follows from a direct computation. For the reader, we show the detail as follows.
By taking $c_v=c=0$ and noting $\a_v(-s)=\a_v(s)$, we have
{\allowdisplaybreaks	\begin{align*}
	\Upsilon_v^{(z)}(\a_v) = & \frac{1}{2\pi i}\left(\int_{0i}^{\frac{2 \pi}{\log q_v}i}+\int_{-\frac{2\pi}{\log q_v}i}^{0i}\right)\frac{-q_v^{-s/2}q_v^{-1/2}}{1-q_v^{-s}q_v^{-(z+1)/2}}\a_v(s)\times\frac{\log q_v}{2}(q_v^{(1+s)/2}-q_v^{(1-s)/2})ds \\
	= & \frac{1}{2\pi i}\int_{0i}^{\frac{2 \pi}{\log q_v}i}\left(\frac{-q_v^{-s/2}q_v^{-1/2}}{1-q_v^{-s}q_v^{-(z+1)/2}}+\frac{q_v^{-s/2}q_v^{-1/2}}{1-q_v^{s}q_v^{-(z+1)/2}}\right)\a_v(s)\\
	& \times \frac{\log q_v}{2}(q_v^{(1+s)/2}-q_v^{(1-s)/2})ds\\
	=&- \frac{1}{2\pi i}\int_{-2}^{2}\frac{(1+q_v^{(z+1)/2})i\sqrt{4-x^2}}{(q_v^{(z+1)/4}+q_v^{-(z+1)/4})^2-x^2}\,X_{m_v}(x)dx.
	\end{align*}
}The last equality is deduced by the variable change $x = q_v^{-s/2} + q_v^{s/2}$.
	By applying \cite[\S2.3]{Serre} for $q=q_v^{\frac{z+1}{4}}$, we obtain the assertion.
\end{proof}

Define $f_\ga \in C([-2,2]^S)$ by $f_\ga((x_v)_{v \in S}) = \prod_{v \in S}X_{n_v}(x_v)$
for an integral ideal $\ga=\prod_{v \in S}(\gp_v\cap\go)^{n_v}$.
By Theorem \ref{quantitative ver of equidist}, we have the following.
\begin{cor}\label{equidist of Lambda}
	Fix any $\s \in (0,1]$ and suppose $\underline{l}>\s+3$.
For $0<z\le 1$, we have
\begin{align*}
\frac{\Lambda_{\gn}^{(z)}(f_\ga)}{\Lambda_{\gn}^{(z)}(1)}
= & \delta(\ga=\square) \nr(\ga)^{(-z-1)/4}
+\Ocal_{\e,\e'}(\nr(\ga)^{(\underline{l}/2+d_F-1-(\s+1)/2)/d_F-\e/d_F+\e'}\nr(\gn)^{-\delta_1+\e})
\end{align*}
for small $\e,\e'>0$, where the implied constant is independent of $\gn$, $\ga$ and $z \in (0,1]$.

When $z=0$, we have
	\begin{align*}
\frac{\Lambda_{\gn}^{(0)}(f_\ga)}{\Lambda_{\gn}^{(0)}(1)}
=\, &
\delta(\ga=\square)\nr(\ga)^{-1/4}-\frac{1}{2}(\log \nr(\ga))
\delta(\ga=\square)\nr(\ga)^{-1/4}C_l^{(0)} \frac{\Res_{z=1}\zeta_{F,\fin}(z)}{{\rm M}(\gn)D(\gn)}\\
& +\Ocal_{\e,\e'}\bigg(\nr(\ga)^{(\underline{l}/2+d_F-1-(\s+1)/2)/d_F-\e/d_F+\e'}\frac{\nr(\gn)^{-\delta_1+\e}}{{\rm M}(\gn)}\bigg),
\end{align*}
where
\begin{align*} D(\gn) := & \frac{1}{{\rm M}(\gn)}
2\,\Res_{z=1}\zeta_{F,\fin}(z) \,C_l^{(0)}
\bigg\{ d_F\left(-\frac{1}{2}\log 2-\frac{3}{4}\log \pi+\frac{1}{4}\psi\left(\frac{3}{4}\right)\right) \\
&+\frac{1}{2}\sum_{v \in \Sigma_{\infty}} \psi(l_v-\tfrac{1}{2})\bigg\} \\
&+\frac{1}{{\rm M}(\gn)}2\,{\rm CT}_{z=1}\zeta_{F, \fin}(z) \,C_l^{(0)}
+\Res_{z=1}\zeta_{F, \fin}(z)
\left(1 +\frac{3\log D_F}{2{\rm M}(\gn)} \right)C_l^{(0)}.
\end{align*}
\end{cor}
\begin{proof}
We follow the proof of \cite[Proposition 3.1]{KnightlyReno}.
Denote by $F_\ga$ the main term \eqref{main of lambda z}
(resp.\ the sum of \eqref{main of lambda 01}, \eqref{main of lambda 02}, \eqref{main of lambda 03} and \eqref{main of lambda 04}) for $z \neq 0$
(resp.\ $z=0$)
in Theorem \ref{quantitative ver of equidist}, and put $E_\ga=\Lambda_{\gn}(f_\ga)-F_\ga$.
Then, we see \begin{align*}	\frac{\Lambda_\gn^{(z)}(f_\ga)}{\Lambda_{\gn}^{(z)}(1)}=\frac{F_\ga+E_\ga}{F_\go+E_\go}=\frac{F_\ga}{F_\go}+\frac{E_\ga-\tfrac{F_\ga}{F_\go}E_\go}{F_\go+E_\go}.\end{align*}
The first term above yields the main term of the assertion.
Indeed, for $z\neq0$, The explicit form of the main term in the assertion is given by Lemma \ref{explicit unip int}. We estimate the second term of the assertion as
\begin{align*} & \frac{E_\ga-\tfrac{F_\ga}{F_\go}E_\go}{F_\go +E_\go}
\ll \frac{E_\ga+\delta(\ga=\square)\nr(\ga)^{(-z-1)/4}E_\go}{F_\go +E_\go} \\
\ll_{\e,\e'}\, & \frac{1}{F_\go+E_\go}
(\nr(\ga)^{(\underline{l}/2+d_F-1-(\s+1)/2)/d_F-\e/d_F+\e'}\nr(\gn)^{-\delta_1+\e}
+\delta(\ga=\square)\nr(\ga)^{(-z-1)/4}\nr(\gn)^{-\delta_1+\e}
) \\
\ll \, & \nr(\ga)^{(\underline{l}/2+d_F-1-(\s+1)/2)/d_F-\e/d_F+\e'}\nr(\gn)^{-\delta_1+\e},
\end{align*}
where $\| \a \|$ in the error term \eqref{error of lambda z} is negligible by Lemma \ref{esti of alpha}.
This completes the proof for $z\neq 0$.
Next consider the case $z=0$. The first term $\frac{F_\ga}{F_\go}$ is similarly evaluated as
{\allowdisplaybreaks\begin{align*}
&\frac{D(\gn)(-1)^{\#S}\Upsilon^{(0)}(\a)+\frac{2}{{\rm M}(\gn)}\Res_{z=1}\zeta_{F,\fin}(z)C_l^{(0)}(-1)^{\#S}\sum_{v\in S}\frac{d}{dz}\big|_{z=0}\Upsilon_v^{(z)}(\a)\prod_{\substack{w\in S \\ w\neq v}}\Upsilon_w^{(0)}(\a)}
{D(\gn)}
\\
=  \,&
\delta(\ga=\square)\nr(\ga)^{-1/4}-\frac{1}{4}(\log \nr(\ga))
\delta(\ga=\square)\nr(\ga)^{-1/4}C_l^{(0)} \frac{2\Res_{z=1}\zeta_{F,\fin}(z)}{{\rm M}(\gn)D(\gn)}.
\end{align*}
}Here we take $\a\in \Acal(\ga)$ such that $f_\ga=f_\a$. By
$\lim_{\nr(\gn)\rightarrow \infty}(F_\go + E_\go) = \Res_{z=1}\zeta_{F,\fin}(z) C_l^{(0)}$, the error term is estimated as
\begin{align*} \frac{E_\ga-\tfrac{F_\ga}{F_\go}E_\go}{F_\go +E_\go}
\ll_{\e} & \frac{E_\ga+\delta(\ga=\square)\nr(\ga)^{-1/4+\e}E_\go}{F_\go +E_\go}
\ll_{\e} \nr(\ga)^{c+\e}\frac{\nr(\gn)^{-\delta_1+\e}}{{\rm M}(\gn)}.
\end{align*}
Here $\| \a \|$ in \eqref{error of lambda 0} is negligible by Lemma \ref{esti of alpha}. This completes the proof for $z=0$.
\end{proof}

From now, $\gn$ is assumed to be a prime ideal $\gq$ of $\go$
relatively prime to $2\prod_{j=1}^{h}\ga_j$ and $S$.

\begin{lem}\label{esti of old forms}
For $z\in [0,1]$, we have
\begin{align*}I_{\gq}^{(z)}(f_\ga):= & \frac{1}{{\rm M}(\gq)^{\delta(z=0)}}\frac{\nr(\gq)^{(z-1)/2}}{1+\nr(\gq)^{(z+1)/2}}\sum_{\pi \in \Pi_{\rm cus}(l,\go)}W_\gq^{(z)}(\pi)\frac{L(\frac{z+1}{2}, \Sym^2(\pi))}{L(1,\Sym^2(\pi))}f_\ga(\bfx_S(\pi))\\
\ll_{\e} \,& \nr(\gq)^{(-1-z)/2} \nr(\ga)^\e,
\end{align*}
for any $\e>0$, where the implied constant is independent of $\ga$ and $\gq$.
\end{lem}
\begin{proof}
It follows from the inequality
\[0<W_\gq^{(z)}(\pi)= \nr(\gq)^{(1-z)/2}\left\{2-\frac{1-Q(I(|\cdot|_{\gq}^{z/2}))}{1-Q(\pi_\gq)^2}\right\}\le 2 \,\nr(\gq)^{(1-z)/2}\]
for any $\pi \in \Pi_{\rm cus}(l,\go)$ and the estimate $\sup_{x \in [-2,2]}|X_m(x)|\ll m+1$.
\end{proof}
Let us define
$\Lambda_\gq^{(z),*}(f)$ for any $f \in C([-2,2]^S)$
similarly to $\Lambda_\gq^{(z)}(f)$ but restricting the range of the summation from $\Pi_{\rm cus}(l,\gq)$
to $\Pi_{\rm cus}^*(l,\gq)$.
See \S\ref{Densities weighted by special values of symmetric square $L$-functions} for the definition of $\Pi_{\rm cus}^*(l,\gq)$.
\begin{cor}\label{equidist of lambda new part}
	Let $\ga=\prod_{v \in S}(\gp_v\cap\go)^{m_v}$ be an ideal of $\go$.
	Let $\gq$ vary in the set of prime ideals of $\go$ relatively prime to
	$2\ga\prod_{j=1}^{h}\ga_j$.
For $0<z\le 1$, we have
\begin{align*}
\frac{\Lambda_{\gq}^{(z),*}(f_\ga)}{\Lambda_{\gq}^{(z),*}(1)}
= & \delta(\ga=\square) \nr(\ga)^{(-z-1)/4}
+\Ocal_{\e,\e'}(\nr(\ga)^{(\underline{l}/2+d_F-1-(\s+1)/2)/d_F-\e/d_F+\e'}\nr(\gq)^{-\delta_1+\e})
\end{align*}
for any small $\e,\e'>0$, where the implied constant is independent of $\gn$, $\ga$ and $z$.
When $z=0$, we have
\begin{align*}
\frac{\Lambda_{\gq}^{(0),*}(f_\ga)}{\Lambda_{\gq}^{(0),*}(1)}
=\, &
\delta(\ga=\square)\nr(\ga)^{-1/4}
-\delta(\ga=\square)\frac{1}{2}\nr(\ga)^{-1/4}\frac{\log \nr(\ga)}{\log \nr(\gq)}
\left\{1+\Ocal\left(\frac{1}{\log \nr(\gq)}\right)\right\}
\\
& +\Ocal_{\e,\e'}\left(\nr(\ga)^{(\underline{l}/2+d_F-1-(\s+1)/2)/d_F-\e/d_F+\e'}\frac{\nr(\gq)^{-\delta_1+\e}}{\log\nr(\gq)}\right),
\end{align*}
where the implied constant is independent of $\gn$ and $\ga$.
\end{cor}
\begin{proof}Invoking
$\Lambda_{\gq}^{(z),*}(f_\ga)=\Lambda_{\gq}^{(z)}(f_\ga) - I_\gq^{(z)}(f_\ga),$
the same proof of Corollary \ref{equidist of Lambda} goes through with the aid of Lemma
\ref{esti of old forms}.
We remark \[\frac{C_{l}^{(0)}\Res_{z=1}\zeta_{F,\fin}(z)}{D(\gq)}=1+\Ocal\left(\frac{1}{{\rm M}(\gq)}\right)=1+\Ocal\left(\frac{1}{\log \nr(\gq)}\right)\]
as $\nr(\gq)\rightarrow \infty$.
\end{proof}

\section{Weighted distributions of low-lying zeros}
\label{Weighted distributions of low-lying zeros}

\subsection{Symmetric power $L$-functions}
\label{Symmetric power $L$-functions}
Let $F$ be a totally real number field such that $2 \in \QQ$ is completely splitting in $F$ as in \S\ref{Refinements of equidistributions weighted by symmetric square L-functions}.
Let $l= (l_v)_{v \in \Sigma_\infty}$ be a family of positive even integers and $\gq$
a prime ideal of $\go$, and fix $r \in \NN$.
For $\pi \in \Pi_{\rm cus}^*(l,\gq)$,
we define the completed symmetric power $L$-function
$L(s,\Sym^r(\pi))=\prod_{v \in \Sigma_\infty\cup\Sigma_\fin}L(s, \Sym^r(\pi_v))$
as in \cite[\S3]{CogdellMichel} (see also \cite[\S2.1.4]{RicottaRoyer}).
First we define the local $L$-factors of $\Sym^r(\pi_v)$ as follows.
For $v \in \Sigma_\infty$, set
\[L(s, \Sym^r(\pi_v)) =\prod_{j=0}^{\frac{r-1}{2}}\Gamma_\RR\left(s+(2j+1)\frac{l_v-1}{2}\right)\Gamma_\RR\left(s+1+(2j+1)\frac{l_v-1}{2}\right)\]
if $r$ is odd, and
\[L(s, \Sym^r(\pi_v))=\Gamma_\RR(s+\mu_{r})\prod_{j=1}^{\frac{r}{2}}\Gamma_\RR\left(s+j(l_v-1)\right)\Gamma_\RR\left(s+1+j(l_v-1)\right)\]
with $\mu_{r}=\delta(r/2 \in 2\NN_0+1) \in \{0,1\}$ if $r$ is even.
For $v \in \Sigma_\fin-S(\gq)$, set
\[L(s,\Sym^r(\pi_v)) = \det \left(1_{r+1}-q_v^{-s}\Sym^r\left(\begin{smallmatrix}
q_v^{-\nu_v(\pi)/2} & 0\\
0 & q_v^{\nu_v(\pi)/2}
\end{smallmatrix}\right)\right)^{-1}= \prod_{j=0}^{r}(1-(q_{v}^{\nu_v(\pi)/2})^{2j-r}q_v^{-s})^{-1},\]
where $1_{r+1}$ is the $(r+1)\times (r+1)$ unit matrix and $\Sym^r : \GL_2(\CC)\rightarrow \GL_{r+1}(\CC)$ is the $r$th symmetric tensor representation.
At $v =\gq$,
the conductor of $\pi_v$ equals $\gp_v=\gq\go_v$ and $\pi_v$ is isomorphic to $\chi_v\otimes {\rm St}_{2}$, where ${\rm St}_{2}$ is the Steinberg representation of $\GL_2(F_v)$ and $\chi_v$ is an unramified character of $F_v^\times$ such that $\chi_v^2$ is trivial. Set
\[L(s,\Sym^r(\pi_v)) = (1-\chi_v(\varpi_v)^r q_v^{-s-r/2})^{-1}.\]
Then, $L(s, \Sym^r(\pi))$ is expected to have an analytic continuation to $\CC$
and the functional equation $L(s, \Sym^r(\pi)) = \epsilon_{\pi,r} (D_F^{r+1}\nr(\gq)^{r})^{1/2-s}L(1-s,\Sym^r(\pi))$, where $\e_{\pi,r}\in\{\pm 1\}$
is defined as
\[\epsilon_{\pi,r}=\begin{cases}
	\{\prod_{v \in \Sigma_\infty} \prod_{j=0}^{(r-1)/2}i^{(2j+1)(l_v-1)+1}\}
(-\chi_\gq(\varpi_\gq)^r)^r
& (r \in 2\NN_0+1),\\
1  & (r \in 2\NN)\\
\end{cases}\]
by \cite[\S3]{CogdellMichel}. The local $L$-factors defined above are compatible with the local Langlands correspondence for $\GL_n$.

Throughout this article, we assume ${\rm Nice}(\pi,r)$ in the introduction for all $\pi \in \Pi_{\rm cus}^{*}(l,\gq)$ and for all prime ideals $\gq$ relatively prime to $2\prod_{j=1}^{h}\ga_j$ for a fixed $l$.
Here we review known results related with the hypothesis $\Nice(\pi,r)$.
In the case $r=1$, this hypothesis is well-known to be true.
In our setting, $\pi$ has non-CM since $\gf_\pi$ is square-free and the central character 
of $\pi$ is trivial.
From this, $\Nice(\pi,r)$ is known for $r=2$ by Gelbart and Jacquet \cite[(9.3) Theorem]{GelbartJacquet},
$r=3$ by Kim and Shahidi \cite[Corollary 6.4]{KimShahidi}
and $r=4$ by Kim \cite[Theorem B and \S7.2]{Kim}.
For any $r\in \NN$, the meromorphy of $L(s, \Sym^r(\pi))$ and its functional equation
can be proved by the potential automorphy of Galois representations $\Sym^r\circ\rho_\pi$,
where $\rho_\pi$ is the Galois representation attached to $\pi$.
See Harris, Shepherd-Barron and Taylor \cite{HSBT} for non-CM elliptic curves over a totally real number field with multiplicative reduction at a finite place,
Gee \cite{Gee} for non-CM elliptic modular forms of weight $3$ with a twisted Steinberg representation at a prime,
Barnet-Lamb, Geraghty, Harris and Taylor \cite{BLGHT}
for non-CM elliptic modular forms of general weight, level and nebentypus,
and Barnet-Lamb, Gee and Geraghty \cite{BLGG} for Hilbert modular
forms.

Automorphy of $\Sym^r(\pi)$ with higher $r$ has been studied and is known in several cases.
When $\pi$ is attached to a holomorphic elliptic cusp form of level $1$,
$\Sym^5(\pi)$ is automorphic by Dieulefait \cite{Dieulefait}.
For some class of totally real number fields $F$
and for any non-CM regular $C$-algebraic irreducible cuspidal automorphic representation $\pi$ of $\GL_2(\AA_F)$,
the lift $\Sym^r(\pi)$ is automorphic when $r=5, 7$ by Clozel and Thorne
\cite[Corollary 1.3]{ClozelThorne2}
and when $r=6,8$ by Clozel and Thorne \cite[Corollary 1.2]{ClozelThorne3}.
As a recent remarkable work, if we restrict our case to elliptic modular forms $F=\QQ$,
the hypothesis $\Nice(\pi,r)$ holds true for all $r\in \NN$, all $\pi \in \Pi_{\rm cus}^*(l,\gq)$ and all prime ideals $\gq$
since $\Sym^r(\pi)$ is automorphic and cuspidal by Newton and Thorne \cite{NewtonThorne}.
They generalized it to \cite{NewtonThorne2}, which covers general levels of elliptic modular forms.

Let $\phi$ be an even Schwartz function on $\RR$ whose Fourier transform $\hat\phi$ is compactly supported.
In the same manner as \cite[Lemma 2.6]{Guloglu} and \cite[Proposition 3.8]{RicottaRoyer},
the explicit formula for $L(s, \Sym^r(\pi))$ \`a la Weil is stated as
{\allowdisplaybreaks\begin{align*}&D(\Sym^r(\pi),\phi) \\
= & \,\hat\phi(0) +\frac{(-1)^{r+1}}{2}\phi(0) -\frac{2}{\log Q(\Sym^r(\pi))} \sum_{v \in \Sigma_{\fin}-S(\gq)}\lambda_\pi(\gp_v^r)\frac{\log q_v}{q_v^{1/2}}
\hat\phi\left(\frac{\log q_v}{\log Q(\Sym^r(\pi))}\right)\\
& -\sum_{m=0}^{r-1}(-1)^{m}\frac{2}{\log Q(\Sym^r(\pi))} \sum_{v \in \Sigma_{\fin}-S(\gq)}\lambda_\pi(\gp_v^{2(r-m)})\frac{\log q_v}{q_v}
\hat\phi\left(\frac{2\log q_v}{\log Q(\Sym^r(\pi))}\right)\\
&+\Ocal\left(\frac{1}{\log Q(\Sym^r(\pi))}\right), \qquad \nr(\gq)\rightarrow \infty.
\end{align*}
}Here $Q(\Sym^r(\pi)):=(\prod_{v \in \Sigma_\infty}l_v^{2\lfloor\frac{r+1}{2}\rfloor})\nr(\gq)^r$ is the analytic conductor of $\Sym^r(\pi)$ and set
\[\lambda_\pi(\gp_v^m)=\sum_{j=0}^{r}(q_v^{-m\nu_v(\pi)/2})^j(q_v^{m\nu_v(\pi)/2})^{r-j},\qquad  m \in \NN.\]
We remark that the explicit formula above is still valid even if $Q(\Sym^r(\pi))$ is replaced with
$Q_r := \nr(\gq)^r$.

We consider the averaged one-level density of low-lying zeros of $L(s, \Sym^r(\pi))$ weighted
by special values $L(\frac{z+1}{2}, \Sym^2(\pi))$.
For $l=(l_v)_{v \in \Sigma_\infty}$ with $\underline{l}:=\min_{v \in \Sigma_\infty} l_v\ge 4$, a prime ideal $\gq$, $z \in [0, \min(1, \s)]$ with a fixed $\s\in (0, \underline{l}-3)$, and
a map $A_\bullet : \Pi_{\rm cus}^*(l,\gq) \rightarrow \CC \ ; \ \pi\mapsto A_\pi$,
set
\[\Ecal_z(A) =\frac{1}{\sum_{\pi\in \Pi_{\rm cus}^{*}(l,\gq)}\frac{L(\frac{z+1}{2},\Sym^2(\pi))}{L(1,\Sym^2(\pi))}}\sum_{\pi\in \Pi_{\rm cus}^*(l,\gq)}\frac{L(\frac{z+1}{2},\Sym^2(\pi))}{L(1, \Sym^2(\pi))} A_\pi.\]

\begin{prop}\label{one level density for z>0}
	For any $r \in \NN$ and $z \in (0,\min(1,\s)]$, set \[\beta_0=\frac{\delta_1}{r\{r(\frac{\underline{l}}{2}+d_F-1-\frac{z+1}{2})\frac{1}{d_F}+\frac{1}{2}\}}>0.\]
where $\delta_1>0$ is a number defined in Theorem \ref{quantitative ver of equidist}.
Then, for any even Schwartz function $\phi$ on $\RR$ with ${\rm supp}(\hat\phi) \subset (-\beta_0,\beta_0)$, we have
	\begin{align*}
	\Ecal_z(D(\Sym^r(\bullet),\phi))
	= \hat\phi(0)+\frac{(-1)^{r+1}}{2}\phi(0)+\Ocal\left(\frac{1}{\log Q_r}\right), \qquad \nr(\gq)\rightarrow \infty.
	\end{align*}
\end{prop}
\begin{proof}We assume  $\beta>0$ and ${\rm supp}(\hat\phi)\subset[-\beta,\beta]$, where $\beta>0$ is suitably chosen later.
We start evaluation from the expression
{\allowdisplaybreaks\begin{align}
& \Ecal_z(D(\Sym^r(\bullet),\phi))\label{formula of E_z} \\
= & \,\hat\phi(0)+\frac{(-1)^{r+1}}{2}\phi(0) +\Ocal\left(\frac{1}{\log Q_r}\right)\notag\\
&-\frac{2}{\log Q_r}
\sum_{v \in \Sigma_\fin-S(\gq)}\Ecal_z(\lambda_\pi(\gp_v^r))\frac{\log q_v}{q_v^{1/2}}\hat\phi\left(\frac{\log q_v}{\log Q_r}\right) \label{M^1}\\
&
-\sum_{m=0}^{r-1}(-1)^{m}\frac{2}{\log Q_r} \sum_{v \in \Sigma_{\fin}-S(\gq)}\Ecal_z(\lambda_\pi(\gp_v^{2(r-m)}))\frac{\log q_v}{q_v}
\hat\phi\left(\frac{2\log q_v}{\log Q_r}\right).\label{M^2}
\end{align}
}as $\nr(\gq)\rightarrow \infty$. We denote the terms \eqref{M^1} and \eqref{M^2} by $M^{(1)}$ and $M^{(2)}$, respectively.
Since $\nr(\gq)$ tends to infinity, we may assume that $\gq$ is relatively prime to	$2\ga\prod_{j=1}^{h}\ga_j$.
Set \[A = (\underline{l}/2+d_F-1-(\s+1)/2)/d_F-\e/d_F+\e'\]
for any fixed small $\e, \e'>0$. Then, Corollary \ref{equidist of lambda new part} yields
{\allowdisplaybreaks\begin{align*}
&M^{(1)} = -\frac{2}{\log Q_r}\sum_{v \in \Sigma_\fin-S(\gq)}\{
\delta(r\in 2\NN)(q_v^r)^{-(z+1)/4}+\Ocal_{\e,\e'}((q_v^r)^{A}\nr(\gq)^{-\delta_1+\e})
\}\frac{\log q_v}{q_v^{1/2}}\hat\phi\left(\frac{\log q_v}{\log Q_r}\right)\\
= & -\frac{2\delta(r\in 2\NN)}{\log Q_r}\sum_{v \in \Sigma_\fin}\frac{\log q_v}{q_v^{1/2+r/4+rz/4}} +\Ocal\left(\frac{1}{\log Q_r}\right)+\Ocal\bigg(\sum_{\substack{v \in \Sigma_\fin-S(\gq) \\ q_v\le Q_r^\beta}}q_v^{rA}\frac{\log q_v}{q_v^{1/2}}\frac{\nr(\gq)^{-\delta_1+\e}}{\log Q_r}\bigg) \\
= & \Ocal\left(\frac{\delta(r\in2\NN)}{\log Q_r}\sum_{v \in \Sigma_\fin}\frac{\log q_v}{q_v^{1+rz/4}}\right) +\Ocal\left(\frac{\nr(\gq)^{r\beta(rA+1/2)}\nr(\gq)^{-\delta_1+\e}}{\log Q_r}\right),
\qquad \nr(\gq)\rightarrow \infty.
\end{align*}
}Here we use the inequality $1/2+r/4+rz/4 \ge 1+rz/4>1$ for $r\ge 2$
and the asymptotics
\begin{align}\label{asymp from PNT}
\sum_{\substack{v \in \Sigma_{\fin} \\ q_v\le x}}q_v^{a}\log q_v =\frac{1}{a+1}x^{a+1}+\Ocal\left(\frac{x^{a+1}}{\log x}\right) =\Ocal(x^{a+1}), \qquad x \rightarrow \infty
\end{align}
for $a>-1$ deduced from the prime ideal theorem and partial summation.
Consequently, the estimate $M^{(1)}\ll\frac{1}{\log Q_r}$ holds as long as $\beta\le \frac{\delta_1-\e}{r(rA+1/2)}$.

Furthermore, we obtain $M^{(2)} \ll \frac{1}{\log Q_r}$ because of the estimate
\begin{align*}
&\sum_{v \in \Sigma_{\fin}-S(\gq)}
\{(q_v^{2(r-m)})^{-(z+1)/4}+\Ocal_{l}((q_v^{r-m})^{A}\nr(\gq)^{-\delta_1+\e})\}
\frac{\log q_v}{q_v}
\hat\phi\left(\frac{\log q_v}{\log Q_r}\right)\\
\ll & \sum_{v \in \Sigma_\fin}\frac{\log q_v}{q_v^{1+\frac{(r-m)(z+1)}{2}}} + (\nr(\gq)^{r\beta})^{rA}\nr(\gq)^{-\delta_1+\e}
\ll 1
\end{align*}
for any $0\le m \le r-1$ by virtue of Corollary \ref{equidist of lambda new part} and \eqref{asymp from PNT}, as long as $\beta\le\frac{\delta_1-\e}{r^2A}$. By removing $\e$ and $\e'$ from two inequalities on $\beta$ as above,
we are done.
\end{proof}

Next let us consider the central value case $z=0$.
\begin{prop}\label{one level density for z=0}
	For any $r \in \NN$ and $z=0$, let $\beta_0>0$ be the same as in Proposition \ref{one level density for z>0} for $z=0$.
Then, for any even Schwartz function $\phi$ on $\RR$ with ${\rm supp}(\phi) \subset (-\beta_0, \beta_0)$, we have
	\begin{align*}
	\Ecal_0(D(\Sym^r(\bullet),\phi))
	= \hat\phi(0)+\frac{(-1)^{r+1}}{2}\phi(0)+\delta_{r,2}\{-\phi(0)+2\int_{-\infty}^{\infty}\hat\phi(x)|x|dx\} +\Ocal\left(\frac{1}{\log \nr(\gq)}\right)
	\end{align*}
as $\nr(\gq)\rightarrow \infty$ with the implied constant independent of $\gq$,
where $\delta_{r,2}:=\delta(r=2)$.
\end{prop}
\begin{proof}
We assume $\beta>0$ and ${\rm supp}(\hat\phi)\subset[-\beta,\beta]$, where $\beta>0$ is suitably chosen later.
The formula \eqref{formula of E_z} is valid for $z=0$, and we
define $M^{(1)}$ and $M^{(2)}$ in the same way as the proof of Proposition
\ref{one level density for z>0}.
With the aid of Corollary \ref{equidist of lambda new part} for $z=0$,
the term $M^{(1)}$ is evaluated as
{\allowdisplaybreaks\begin{align*}
	M^{(1)} = & -\frac{2}{\log Q_r}\sum_{v \in \Sigma_\fin-S(\gq)}\bigg\{
	\delta(r \in 2\NN)q_v^{-r/4}-\frac{\delta(r \in 2\NN)}{2}q_v^{-r/4}\frac{\log (q_v^r)}{\log \nr(\gq)}\\
	& \times \left\{1+\Ocal\left(\frac{1}{\log\nr(\gq)}\right)\right\}
	+	\Ocal_{\e,\e'}\left(q_v^{rA}\frac{\nr(\gq)^{-\delta_1+\e}}{\log \nr(\gq)}\right)
	\bigg\} \frac{\log q_v}{q_v^{1/2}}\hat\phi\left(\frac{\log q_v}{\log Q_r}\right)\\
	= & -2 \delta(r \in 2\NN)\sum_{v \in \Sigma_{\fin}}\hat\phi\left(\frac{\log q_v}{\log Q_r}\right)\frac{\log q_v}{q_v^{1/2+r/4}\log Q_r}\\
&	+\delta(r \in 2\NN)\left\{1+\Ocal\left(\frac{1}{\log\nr(\gq)}\right)\right\}
\frac{r}{\log \nr(\gq)} \sum_{v \in \Sigma_{\fin}}
	\hat\phi\left(\frac{\log q_v}{\log Q_r}\right)\frac{(\log q_v)^2}{q_v^{1/2+r/4}\log Q_r}\\
	& +\Ocal\left(\frac{1}{\log Q_r}\right)+\Ocal_{\e,\e'}\bigg(\frac{1}{\log Q_r}\sum_{\substack{v \in \Sigma_\fin-S(\gq) \\ q_v\le Q_r^\beta}}q_v^{rA}\frac{\log q_v}{q_v^{1/2}}\frac{\nr(\gq)^{-\delta_1+\e}}{\log \nr(\gq)}\bigg)
	\end{align*}
}as $\nr(\gq)\rightarrow \infty$ for any fixed $\e, \e'>0$.
When $r\ge 3$, then $1/2+r/4>1$ is satisfied and hence $M^{(1)}$ is estimated as $\Ocal(\frac{1}{\log \nr(\gq)})$
with the aid of the estimate
\begin{align*}
\sum_{\substack{v \in \Sigma_\fin-S(\gq) \\ q_v\le Q_r^\beta}}q_v^{rA}\frac{\log q_v}{q_v^{1/2}}\nr(\gq)^{-\delta_1+\e}\ll (Q_r^\beta)^{rA+1/2}\nr(\gq)^{-\delta_1+\e}\ll 1
\end{align*}
from the asymptotics \eqref{asymp from PNT} under $\beta \le\frac{\delta_1-\e}{r^2A}$. The case $r=1$ is similarly estimated.
Consequently we have $M^{(1)}\ll \frac{1}{\log \nr(\gq)}$ when $r\neq2$.
Next let us consider the case $r=2$. By $1/2+r/4=1$,
we need two asymptotics for evaluating $M^{(1)}$:
\begin{align*}
\sum_{v \in \Sigma_\fin}\hat\phi\left(\frac{\log q_v}{\log Q_2}\right)\frac{\log q_v}{q_v\log Q_2}
=& \frac{1}{2}\phi(0) +\Ocal\left(\frac{1}{\log Q_2}\right), \qquad \nr(\gq)\rightarrow \infty,
\end{align*}
\begin{align*}
\sum_{v \in \Sigma_\fin}\hat\phi\left(\frac{\log q_v}{\log Q_2}\right)\frac{(\log q_v)^2}{q_v(\log Q_2)^2}
=& \frac{1}{2}\int_{-\infty}^{\infty}\hat\phi(x)|x|dx+\Ocal\left(\frac{1}{\log Q_2}\right), \qquad \nr(\gq)\rightarrow \infty.
\end{align*}
These are proved by the prime ideal theorem and partial summation (cf.\ \cite[Lemma 4.4 i), iii)]{Miller}).
From these, a direct computation
yields
{\allowdisplaybreaks\begin{align*}
M^{(1)} = & -\phi(0) +\frac{\log Q_2}{\log\nr(\gq)}\int_{-\infty}^{\infty} \hat\phi(x)|x|dx+\Ocal\left(\frac{1}{\log\nr(\gq)}\right)
+\Ocal_{\e,\e'}\left(\frac{1+\nr(\gq)^{2\beta(2A+1/2)-\delta_1+\e}}{\log Q_2}\right) \\
= & -\phi(0)+2(1+\nr(\gq)^{-1/2})\int_{-\infty}^\infty \hat\phi(x)|x|dx + \Ocal_{\e,\e'}\left(\frac{1}{\log \nr(\gq)}\right)
\end{align*}
}as $\nr(\gq)\rightarrow \infty$ under $\beta\le \frac{\delta_1-\e}{2(2A+1/2)}$.
This gives the evaluation of $M^{(1)}$ for $r=2$.
The term $M^{(2)}$ for any $r \in \NN$ is estimated by $\Ocal(\frac{1}{\log \nr(\gq)})$ similarly to Proposition \ref{one level density for z>0} under $\beta\le \frac{\delta_1-\e}{2(2A+1/2)}$.
Thus we are done by removing $\e$ and $\e'$ from the inequalities on $\beta$ above.
\end{proof}

Theorem \ref{main: weighted one level density for z}
is proved by Propositions \ref{one level density for z>0} and \ref{one level density for z=0}.
The explicit form of $\beta_2$ in Theorem \ref{main: weighted one level density for z}
is given by $\beta_0$ in Proposition \ref{one level density for z>0}
and $\delta_1=\min(\delta,\delta')-1/2$ in Theorem \ref{quantitative ver of equidist}
with the aid of Lemmas \ref{hyperbolic term} and \ref{elliptic term}.

As a remark,
if we specialize the parameter $z$ to $z=1$,
Theorem \ref{main: weighted one level density for z} becomes
a formula similar to
\cite[Theorem 11.5]{ShinTemplier} for $G=\PGL_2$ and $\Sym^r : {}^L\PGL_2\rightarrow \GL_{r+1}(\CC)$,
although the principal congruence subgroup $\Gamma(\gq)$ is considered there.
Now Hypotheses 11.2 and 11.4 in \cite{ShinTemplier} are satisfied
and Hypothesis 10.1 in \cite{ShinTemplier} is replaced with ${\bf \rm Nice}(\pi,r)$ in our setting.
Furthermore, the Frobenius-Schur indicator $s(\Sym^r)$ of $\Sym^r : \SL_2(\CC)\rightarrow \GL_{r+1}(\CC)$ is equal to $(-1)^r$ (cf.\ \cite[Exercise 11.33]{FultonHarris}).

\appendix
\section{Comparison of the ST trace formula with Zagier's formula}
\label{Comparison of the ST trace formula with Zagier's formula}
Appendices \ref{Comparison of the ST trace formula with Zagier's formula}, \ref{Comparison of the ST trace formula with Takase's formula} and \ref{Comparison of the ST trace formula with Mizumoto's formula} are independent of interest in our main results in this article.
When \cite{SugiyamaTsuzuki2018} was reviewed, the anonymous referee
said to the author that it is non-trivial to recover known formulas due to Zagier \cite{Zagier}, due to Mizumoto \cite{Mizumoto} and due to Takase \cite{Takase},
from the trace formula in Theorem \ref{JZtrace} by Tsuzuki and the author \cite{SugiyamaTsuzuki2018} what is called {\it the ST trace formula} in this article.
Thus, comparison among these trace formulas above is meaningful and helpful.
In Appendix \ref{Comparison of the ST trace formula with Zagier's formula}, we prove the compatibility of the ST trace formula with Zagier's formula.
Takase's and Mizumoto's formulas are considered
in Appendix \ref{Comparison of the ST trace formula with Takase's formula} and Appendix \ref{Comparison of the ST trace formula with Mizumoto's formula}, respectively.

In what follows, notation is the same as in \cite{SugiyamaTsuzuki2018} (see also \S \ref{Refinements of equidistributions weighted by symmetric square L-functions}).
We use the functional equation $\zeta_{F}(s)=D_F^{1/2-s}\zeta_F(1-s)$
of the completed Dedekind zeta function of $F$
and the duplication formula $\Gamma(2z)=2^{2z-1}\pi^{-1/2}\Gamma(z)\Gamma(z+1/2)$,
which is also stated as $\Gamma_\RR(s)\Gamma_\RR(s+1)=\Gamma_\CC(s)$,
without notice.
For the integer ring $\go$ of a number field and
$m \in \go$, the symbol $m=\square$ means $m=n^2$ for some $n \in \go-\{0\}$.

\subsection{Zagier's formula}
We compare the ST trace formula with Zagier's trace formula.
For a positive even integer $k$, let $B_k$ be an orthogonal basis of $S_k(\Gamma_0(1))$ consisting of primitive forms.
For $f \in B_k$, let $a_f(m)$ be the $m$th Fourier coefficient of $f$ at the cusp $i\infty$.
For a discriminant $\Delta \in \ZZ$, $L(s, \Delta)$
is the $L$-function associated with binary quadratic forms with discriminant $\Delta$
divided by the Riemann zeta function (cf.\ \cite[p.109--110]{Zagier}).
The quantity $I_k(\Delta,t;s)$ denotes the integral defined in \cite[p.110]{Zagier}.
Then, Zagier's formula \cite[Theorem 1]{Zagier} is stated as follows.
\begin{thm}[Zagier's formula \cite{Zagier}]\label{Zagier formula}
	Assume $k\ge 4$ and $m \in \NN$. For $2-k<\Re(s)<k-1$, i.e., $|\Re(2s-1)|<2k-3$, we have
	\begin{align*}c_m(s):=&\sum_{f \in B_k}\frac{(-1)^{k/2}\pi}{2^{k-3}(k-1)}
		\frac{\Gamma(s+k-1)}{(4\pi)^{s+k-1}}\frac{L_{\fin}(s, \Sym^2(f))}{(f,f)}\\
		=& m^{k-1}\sum_{t \in \ZZ}\{I_k(t^2-4m,t;s)+I_k(t^2-4m,-t;s)\}L(s,t^2-4m)\\
		&+\delta(m=\square)(-1)^{k/2}\frac{\Gamma(k+s-1)\zeta(2s)}{2^{2s+k-3}\pi^{s-1}\Gamma(k)}m^{\frac{k-1}{2}-\frac{s}{2}}.
	\end{align*}
Here $L_{\fin}(s, \Sym^2(f))$ is the symmetric square $L$-function attached to $f$ with all archimedean local $L$-factors removed.
\end{thm}

To prove Zagier's formula from the ST trace formula, 
we consider the case where $F=\QQ$, $l=k$, $\fn=\ZZ$, $S=S(m\ZZ)=\{p :{\rm prime}
\ | \ \ord_p(m)>0\}$,
$\a(\bfs)=\otimes_{p \in S}X_{n_p}(p^{-s_p/2}+p^{s_p/2})$ with
$n_p = \ord_p(m)$,
and $z=2s-1$
throughout Appendix \ref{Comparison of the ST trace formula with Zagier's formula}.
We assume $|\Re(2s-1)|<k-3$ to use the ST trace formula, which is narrower than Zagier's range $|\Re(2s-1)|<2k-3$.
\subsection{Cuspidal terms of Zagier's formula}
\label{Cuspidal terms of Zagier's formula}

As for the spectral side, we obtain
\allowdisplaybreaks{
\begin{align*}
	(-1)^{\#S}C(k,\ZZ)\II_{\rm cusp}^0(\ZZ|\a,z) = & (-1)^{\#S}2^{-1}\frac{4\pi}{k-1}\sum_{\pi \in \Pi_{\rm cus}(k,\ZZ)}\frac{L(\frac{z+1}{2}, \Sym^2(\pi))}{L(1, \Sym^2(\pi))}\a(\nu_S(\pi))\\
	= &  (-1)^{\#S}\frac{2\pi}{k-1}\sum_{f \in B_k}\frac{L(\frac{z+1}{2}, \Sym^2(f))}{L(1, \Sym^2(f))}\frac{a_f(m)}{m^{\frac{k-1}{2}}}.
\end{align*}
}As for Zagier's formula, the left-hand side of Zagier's formula is described as
\[c_m(s)=(-1)^{k/2}\frac{1}{2^{s+k-2}\pi^{-s/2-1/2}\Gamma(\frac{s+1}{2})} \times \frac{2\pi}{k-1}\sum_{f \in B_k}\frac{L(s,\Sym^2(f))}{L(1, \Sym^2(f))}a_f(m),\]
where we use $L(1,\Sym^2(f))=2^k(f,f)$ (cf.\ \eqref{norm of new form}).
The ST trace formula yields
\begin{align}\label{cuspidal of Zagier}
	c_m(s)= & \frac{(-1)^{k/2}(-1)^{\#S}m^{\frac{k-1}{2}}}{2^{s+k-2}\pi^{-s/2-1/2}\Gamma(\frac{s+1}{2})}
	\times(-1)^{\#S}C(l,\ZZ)\II_{\rm cusp}^0(\ZZ|\a,z) \\
	= & \frac{(-1)^{k/2}(-1)^{\#S}m^{\frac{k-1}{2}}}{2^{s+k-2}\pi^{-s/2-1/2}\Gamma(\frac{s+1}{2})}\{\JJ_{\rm unip}^0(\ZZ|\a,z)+\JJ_{\rm unip}^0(\ZZ|\a,-z)+\JJ_{\rm hyp}^0(\ZZ|\a,z)+\JJ_{\rm ell}^0(\ZZ|\a,z)\}.\notag
\end{align}

\subsection{Unipotent terms of Zagier's formula}
\label{Unipotent terms of Zagier's formula}
Let $\zeta(s)$ denote the Riemann zeta function. Then, $\zeta_\QQ(s)=\Gamma_\RR(s)\zeta(s)$ holds.
By definition, we have
\begin{align*}
	\JJ_{\rm unip}^0(\ZZ|\a,z) = &
	\frac{(-1)^{k/2}\Gamma(k+s-1)\zeta(2s)}{2^{2s+k-3}\pi^{s-1}\Gamma(k)}\Upsilon^{(z)}(\a)
	\times(-1)^{k/2}2^{s+k-2}\pi^{-s/2-1/2}\Gamma(\tfrac{s+1}{2}),
\end{align*}
where we note $k=l$ and $s=\frac{z+1}{2}$.
The term $\JJ_{\rm unip}^0(\ZZ|\a, -z)$ is transformed into
\allowdisplaybreaks{\begin{align*}
&	\frac{(-1)^{k/2}\Gamma(k-s)\zeta(2-2s)}{2^{-2s+k-1}\pi^{-s}\Gamma(k)}\Upsilon^{(-z)}(\a)
		\times (-1)^{k/2}2^{s+k-2}\pi^{-s/2-1/2}\Gamma(\tfrac{s+1}{2}) 2^{1-2s}\pi^{s-1/2}\frac{\Gamma(\frac{2-s}{2})}{\Gamma(\frac{s+1}{2})}\\
		= & \frac{(-1)^{k/2}\Gamma(k-s)\zeta(2s-1)}{2^{-2s+k-1}\pi^{-s}\Gamma(k)}\Upsilon^{(-z)}(\a)
		\times (-1)^{k/2}2^{s+k-2}\pi^{-s/2-1/2}\Gamma(\tfrac{s+1}{2})\\
		& \times 2^{1-2s}\pi^{1-s}\frac{\Gamma(s-1/2)}{\Gamma(\frac{s+1}{2})}\frac{\Gamma(1-s/2)}{\Gamma(1-s)}\\
		= &
		\frac{(-1)^{k/2}\Gamma(k-s)\zeta(2s-1)}{2^{-2s+k-1}\pi^{-s}\Gamma(k)}\frac{\Gamma(s-1/2)}{\Gamma(\frac{s+1}{2})\Gamma(\frac{1-s}{2})}\Upsilon^{(-z)}(\a)
		\times (-1)^{k/2}2^{s+k-2}\pi^{-s/2-1/2}\Gamma(\tfrac{s+1}{2})\\
		& \times 2^{1-2s}\pi^{1-s}\frac{\Gamma(\frac{1-s}{2})\Gamma(\frac{1-s}{2}+\frac{1}{2})}{\Gamma(1-s)}\\
		= &
		\frac{(-1)^{k/2}\Gamma(k-s)\zeta(2s-1)}{2^{-2s+k-1}\pi^{-s}\Gamma(k)}\frac{\Gamma(s-1/2)}{\Gamma(\frac{s+1}{2})\Gamma(\frac{1-s}{2})}\Upsilon^{(-z)}(\a)
		\times (-1)^{k/2}2^{s+k-2}\pi^{-s/2-1/2}\Gamma(\tfrac{s+1}{2}) 2^{1-s}\pi^{3/2-s}\\
		= &
		\frac{(-1)^{k/2}\Gamma(k-s)\zeta(2s-1)}{\Gamma(k)}2^{s-k+2}\pi^{3/2}\frac{\Gamma(s-1/2)}{\Gamma(\frac{s+1}{2})\Gamma(\frac{1-s}{2})}\Upsilon^{(-z)}(\a)
		\,\times (-1)^{k/2}2^{s+k-2}\pi^{-s/2-1/2}\Gamma(\tfrac{s+1}{2}).
	\end{align*}
}As a consequence,
from two formulas above, we obtain
\begin{align}\label{unipotent of Zagier}
	& \frac{(-1)^{k/2}(-1)^{\#S}m^{\frac{k-1}{2}}}{2^{s+k-2}\pi^{-s/2-1/2}\Gamma(\frac{s+1}{2})}\{\JJ_{\rm unip}^0(\ZZ|\a,z)+\JJ_{\rm unip}^0(\ZZ|\a,-z) \}\\
	= &  m^{\frac{k-1}{2}}(-1)^{\#S}\Upsilon^{(z)}(\a)\frac{(-1)^{k/2}\Gamma(k+s-1)\zeta(2s)}{2^{2s+k-3}\pi^{s-1}\Gamma(k)} \notag \\
  &	+2\times m^{\frac{k-1}{2}}(-1)^{\#S}\Upsilon^{(-z)}(\a)\pi^{3/2}\frac{(-1)^{k/2}}{\Gamma(\frac{s+1}{2})\Gamma(\frac{1-s}{2})}
	\frac{\Gamma(s-1/2)\Gamma(k-s)}{\Gamma(k)}\zeta(2s-1)2^{s-k+1}.\notag
\end{align}
The first term of the right-hand side corresponds to the second term in Zagier's formula,
and the second term of the right-hand side corresponds to
the sum of two terms for $t^2-4m=0$ (i.e., $t=\pm\sqrt{4m}$) in Zagier's formula given by
\[2\times m^{k-1} \delta(m=\square)\pi^{3/2} \frac{(-1)^{k/2}}{\Gamma(\frac{s+1}{2})\Gamma(\frac{1-s}{2})}
\frac{\Gamma(s-1/2)\Gamma(k-s)}{\Gamma(k)}\zeta(2s-1)
2^{s-k+1}m^{(s-k)/2}\]
(cf.\ \cite[(13)]{Zagier}). Note $(-1)^{\#S}\Upsilon^{(z)}(\a) = \delta(m=\square)\,m^{-s/2}$
and $(-1)^{\#S}\Upsilon^{(-z)}(\a) = \delta(m=\square)\,m^{(s-1)/2}$ by Lemma \ref{explicit unip int}.

\subsection{Hyperbolic terms of Zagier's formula}
\label{Hyperbolic terms of Zagier's formula}
We relate integrals of the test function $\a$ with terms such that $t^2-4m=\square$ in the sum of Zagier's formula.

Let $\Delta$ be an integer such that $\Delta\not=0$ and $\Delta\equiv 0,1\pmod{4}$. Then $\Delta=f^2D$ holds for $f\in \NN$ and a fundamental discriminant $D\in \ZZ$. Set
\[
B_\Delta(s) = \sum_{d|f}\mu(d)\chi_{D}(d)d^{-s}\sigma_{1-2s}\left(\tfrac{f}{d}\right), \quad s\in \CC,
\]
where $d$ runs over all positive divisors of $f$, $\mu$ is the M\"obius function, $\chi_{D}$ is the Kronecker character for $D$, and $\sigma_{1-2s}(d)=\sum_{c|d}c^{1-2s}$. 
\begin{lem} \label{explicit of B_Delta}
	Set $z=2s-1$. Then we have
	\begin{align*}
		B_{\Delta}(s)=f^{-\frac{z+1}{2}}\prod_{p|f}
		\left(\frac{\zeta_{p}(z)}{L_p\left(\tfrac{z+1}{2},\chi_{D}\right)}|f|_p^{-\frac{z+1}{2}}+\frac{\zeta_{p}(-z)}{L_p\left(\tfrac{-z+1}{2},\chi_{D}\right)}|f|_p^{-\frac{-z+1}{2}}\right),
	\end{align*}
where $\zeta_p(s)$ and $L_p(s, \chi_D)$ are the local $L$-factors 
of $\zeta(s)$ and of $L(s,\chi_D)$ at $p$, respectively.
\end{lem}
\begin{proof} This is stated in \cite[p.8]{SugiyamaTsuzuki2019}, where the proof is omitted. The proof here is based on a personal communication of the author with Masao Tsuzuki.
	Let $f=\prod_{p}p^{\nu_p}$ be the prime decomposition and set $f_1=\prod_{p\in S_1}p$ and $f_2=\prod_{p\in S_2}p^{\nu_p}$ with $S_1=\{p \ | \ \nu_p=1\}$ and $S_2=\{p \ | \ \nu_p\geq 2\}$, respectively. Then $f_1$ and $f_2$ are relatively prime and $f=f_1f_2$. We have
	\allowdisplaybreaks{\begin{align*}
		B_{\Delta}(s)&=\sum_{d|f}\mu(d)\chi_D(d)d^{-s}\prod_{p|\frac{f}{d}}\frac{1-p^{1-2s}|f/d|_p^{2s-1}}{1-p^{1-2s}}
		\\
		&=\sum_{d_1|f_1}\sum_{d_2|f_2}\mu(d_1d_2)\chi_{D}(d_1d_2)(d_1d_2)^{-s}
		\prod_{p|\frac{f_1}{d_1}} \frac{1-p^{1-2s}|f_1/d_1|_p^{2s-1}}{1-p^{1-2s}} \prod_{p|\frac{f_2}{d_2}}\frac{1-p^{1-2s}|f_2/d_2|_p^{2s-1}}{1-p^{1-2s}}
		\\
		&=\sum_{d_1|f_1}\sum_{d_2|f_2}\mu(d_1)\mu(d_2)\chi_{D}(d_1)\chi_D(d_2)d_1^{-s}d_2^{-s}\prod_{p|\frac{f_1}{d_1}} 
		\frac{1-p^{1-2s}|f_1|_p^{2s-1}}{1-p^{1-2s}}
		\\
		&\qquad \times 
		\prod_{p|{d_2}}\frac{1-|f_2|_p^{2s-1}}{1-p^{1-2s}}
		\prod_{\substack{p\notdivide{d_2} \\ p|\frac{f_2}{d_2}}}
		\frac{1-p^{1-2s}|f_2|_p^{2s-1}}{1-p^{1-2s}}.
	\end{align*}
}To have the last equality, we note that due to the presence of $\mu$, the numbers $d_1$ and $d_2$ are restricted to square-free divisors, and whence $|d_1|_p=1$ holds for $p|\frac{f_1}{d_1}$, and $p|d_2$ implies $|d_2|_p=p^{-1}$ and $p|\frac{f_2}{d_2}$. Thus $B_{\Delta}(s)$ becomes the product $B_1(s)B_2(s)$ with
	\begin{align*}
		B_1(s)&=\sum_{d_1|f_1}\mu(d_1)\chi_{D}(d_1)d_1^{-s} \prod_{p|\frac{f_1}{d_1}} \frac{1-p^{1-2s}|f_1|_p^{2s-1}}{1-p^{1-2s}}, \\
		B_2(s)&=\sum_{d_2|f_2}\mu(d_2)\chi_{D}(d_2)d_2^{-s}
		\prod_{p|d_2}\frac{1-|f_2|_p^{2s-1}}{1-p^{1-2s}} 
		\prod_{\substack{p\notdivide d_2 \\ p|\frac{f_2}{d_2}}} \frac{1-p^{1-2s}|f_2|_p^{2s-1}}{1-p^{1-2s}}.
	\end{align*}
	Since $d_1$ and $f_1/d_1$ are relatively prime for each $d_1|f_1$, we have
	\begin{align*}
		B_1(s)&=\prod_{p|f_1}\left( -\chi_{D}(p)p^{-s}+\frac{1-p^{1-2s}|f_1|_p^{2s-1}}{1-p^{1-2s}}\right)
		\\&=
		\prod_{p|f_1} \left(\frac{1-\chi_{D}(p)p^{1-2s}}{1-p^{1-2s}}+\frac{1-\chi_{D}(p)p^{s-1}}{1-p^{2s-1}}|f_1|_p^{2s-1} \right)
	\end{align*}
	with the aid of the relation $|f_1|_p=p^{-1}$ for $p|f_1$. Similarly, we have
	\begin{align*}
		B_2(s)&=\prod_{p|f_2}\left(-\chi_{D}(p)p^{-s}\frac{1-|f_2|_p^{2s-1}}{1-p^{1-2s}}+\frac{1-p^{1-2s}|f_2|_p^{2s-1}}{1-p^{1-2s}}\right)
		\\
		&=\prod_{p|f_2}\left(\frac{1-\chi_{D}(p)p^{1-2s}}{1-p^{1-2s}}+\frac{1-\chi_{D}(p)p^{s-1}}{1-p^{2s-1}}|f_2|_p^{2s-1}\right).
	\end{align*}
	The remaining part of the proof is a straightforward calculation. 
\end{proof}

Let us consider the hyperbolic term
\begin{align*}\JJ_{\rm hyp}^0(\ZZ|\a,z)
	= & \frac{1}{2}\zeta\left(\tfrac{1+z}{2}\right)\sum_{a \in \ZZ(S)_+^\times-\{1\}}
	\bfB_\ZZ^{(z)}\left(\a \bigg| 1; \tfrac{a}{(a-1)^2}\ZZ\right)\,\Gamma_\RR\left(\tfrac{1+z}{2}\right)\Ocal_\infty^{+,(z)}\left(\tfrac{a+1}{a-1}\right).
\end{align*}
We remark that $\hat{\bf \Scal}_v^{\Delta,(z)}(\a_v;a(a-1)^{-2})$
is computed for a general number field in the same way as
\cite[Lemma 2.6]{SugiyamaTsuzuki2019} as follows.
\begin{lem}\label{explicit of hatS}Let $v$ be a finite place of a number field $F$.
	Set $d\mu_v(s_v)=\frac{\log q_v}{2}(q_v^{\frac{s_v+1}{2}}-q_v^{\frac{1-s_v}{2}})ds_v$.
	For $\Delta \in F^\times$, $n\in \NN_0$ and $a \in F^\times$, when $|a|_v\le 1$, we have
	\begin{align*}
		&\frac{-1}{2\pi i}\int_{L(c)}(-q_v^{-\frac{s_v+1}{2}})\frac{\zeta_{F_v}(s_v+\frac{z+1}{2})\zeta_{F_v}(s_v+\frac{-z+1}{2})}
		{L_{F_v}(s_v+1,\varepsilon_\Delta)}|a|_v^{\frac{s_v+1}{2}}X_n(q_v^{-s_v/2}+q_v^{s_v/2})d\mu_v(s_v) \\
		= & \delta(n-\ord_v(a) \in 2\NN_0)q_v^{-n/2}\left\{ \frac{\zeta_{F_v}(-z)q_v^{(n-\ord_v(a))(-z+1)/4}}{L_{F_v}(\frac{-z+1}{2},\varepsilon_\Delta)} +\frac{\zeta_{F_v}(z)q_v^{(n-\ord_v(a))(z+1)/4}}{L_{F_v}(\frac{z+1}{2}, \varepsilon_\Delta)}\right\}.
	\end{align*}
	When $|a|_v>1$, we have
	\begin{align*}
		&\frac{-1}{2\pi i}\int_{L(c)}(-q_v^{-\frac{s_v+1}{2}})\left\{ \frac{\zeta_{F_v}(-z)\zeta_{F_v}(s_v+\frac{z+1}{2})}
		{L_{F_v}(\frac{-z+1}{2},\varepsilon_\Delta)}|a|_v^{(-z+1)/4}+ \frac{\zeta_{F_v}(z)\zeta_{F_v}(s_v+\frac{-z+1}{2})}
		{L_{F_v}(\frac{z+1}{2},\varepsilon_\Delta)} |a|_v^{(z+1)/4} \right\} \\
		& \times X_n(q_v^{-s_v/2}+q_v^{s_v/2})d\mu_v(s_v) \\
		= &
		\delta(n \in 2\NN_0)q_v^{-n/2}\left\{ \frac{\zeta_{F_v}(-z)q_v^{(n-\ord_v(a))(-z+1)/4}}{L_{F_v}(\frac{-z+1}{2},\varepsilon_\Delta)} +\frac{\zeta_{F_v}(z)q_v^{(n-\ord_v(a))(z+1)/4}}{L_{F_v}(\frac{z+1}{2}, \varepsilon_\Delta)}\right\}.
	\end{align*}
\end{lem}

To transform the sum over $a \in \ZZ(S)_{+}^{\times}-\{1\}$ into the sum over $t \in \ZZ$ with $t^2-4m = \square$, we need the following fundamental lemma (cf.\ \cite[p.113]{Zagier}).
\begin{lem}
	Let $m$ be a positive integer. Then, the set of $a\in \ZZ(S)_+^{\times}-\{1\}$ such that $a=\frac{t+f}{t-f}$ for some $t \in\ZZ$ and $f \in \NN$ satisfying
	$t^2-4m=f^{2}$ is bijectively corresponds to the set of $a \in \ZZ(S)_+^{\times}-\{1\}$
	such that $a=\frac{d_1}{d_2}$ and $m=d_1d_2$ for some $d_1,d_2\in \NN$ with $d_1\neq d_2$.
\end{lem}
By this lemma and the argument on hyperbolic terms in \cite[\S2.3]{SugiyamaTsuzuki2019},
the sum over $a \in \ZZ(S)^\times_{+}-\{1\}$ is written as the sum over
$t \in \ZZ$ satisfying $t^2-4m=f^2$ for some $f \in \NN$ through $a=\frac{t+f}{t-f}$.
By setting $a=\frac{t+f}{t-f}$, we obtain
\[\frac{a}{(a-1)^2}=\frac{m}{f^2}, \qquad \frac{a+1}{a-1}=\frac{t}{f}.\]
For $m=\prod_{v \in S}p^{n_v}$ with $n_v\ge1$, $S=S(m\ZZ)$, $a = \frac{t+f}{t-f}\in \ZZ(S)_{+}^\times-\{1\}$ and $\Delta=f^2$ with $f \in \NN$, the $\delta$-symbol in $\hat{\bf \Scal}_v^{\Delta,(z)}(\a_v;a(a-1)^{-2})$ caused in Lemma \ref{explicit of hatS}
is removable as in \cite[\S2.3]{SugiyamaTsuzuki2019}.
Hence, Lemmas \ref{explicit of B_Delta} and \ref{explicit of hatS} yield
\begin{align}\label{identity of B}
	& (-1)^{\#S}\bfB_{\ZZ}^{(z)}(\a | 1;\tfrac{a}{(a-1)^2}\ZZ) \\
	= \,& \prod_{v \in S}
	 q_v^{-n_v/2}\left\{ \frac{\zeta_{F_v}(-z)q_v^{(n_v-\ord_v(a/(a-1)^2))(-z+1)/4}}{L_{F_v}(\frac{-z+1}{2},\varepsilon_\Delta)} +\frac{\zeta_{F_v}(z)q_v^{(n_v-\ord_v(a/(a-1)^2))(z+1)/4}}{L_{F_v}(\frac{z+1}{2}, \varepsilon_\Delta)}\right\} \notag \\
	& \times \prod_{v \in \Sigma_\fin-S}\left\{ \frac{\zeta_{F_v}(-z)q_v^{(-\ord_v(a/(a-1)^2))(-z+1)/4}}{L_{F_v}(\frac{-z+1}{2},\varepsilon_\Delta)} +\frac{\zeta_{F_v}(z)q_v^{(-\ord_v(a/(a-1)^2))(z+1)/4}}{L_{F_v}(\frac{z+1}{2}, \varepsilon_\Delta)}\right\} \notag \\
	= \,& m^{-1/2}\prod_{v \in S} \left\{\frac{\zeta_{F_v}(-z)}{L_{F_v}(\frac{-z+1}{2},\varepsilon_\Delta)}\left|\frac{1}{f^2}\right|_v^{(-z+1)/4} +\frac{\zeta_{F_v}(z)}{L_{F_v}(\frac{z+1}{2}, \varepsilon_\Delta)}\left|\frac{1}{f^2}\right|_v^{(z+1)/4}\right\} \notag \\
	& \times \prod_{v \notin S} \left\{\frac{\zeta_{F_v}(-z)}{L_{F_v}(\frac{-z+1}{2},\varepsilon_\Delta)}\left|\frac{m}{f^2}\right|_v^{(-z+1)/4} +\frac{\zeta_{F_v}(z)}{L_{F_v}(\frac{z+1}{2}, \varepsilon_\Delta)}\left|\frac{m}{f^2}\right|_v^{(z+1)/4}\right\}\notag \\
	=\,& m^{-1/2}\times f^{s}B_{\Delta}(s),\notag 
\end{align}
where $F=\QQ$.
For $b=\frac{t}{f} = \frac{t}{\sqrt{\Delta}}$, we observe
\begin{align*}
	\Gamma_\RR(\tfrac{z+1}{2})\Ocal_\infty^{+,(z)}(b) = \,& \Gamma_\RR(\tfrac{1+z}{2})\times \frac{2\pi}{\Gamma(k)}\frac{\Gamma(k+\frac{z-1}{2}) \Gamma(k+\frac{-z-1}{2})}{\Gamma_\RR(\frac{z+1}{2})
		\Gamma_\RR(\frac{1-z}{2})}\delta(|b|>1)(b^2-1)^{1/2}{\mathfrak P}_{\frac{z-1}{2}}^{1-k}(|b|) \\
	= \,& \frac{\Gamma(k-1/2)}{\Gamma(k)\Gamma(1/2)}\frac{2^{2k}\pi m^{k/2}f^{-k}}{\Gamma_\RR(1-s)} \delta(|b|>1)I_{k,s}(|b|)\\
	=\, & \frac{2^{k+s-1}m^{k/2}f^{-s}}{\Gamma_\RR(1-s)\cos(\frac{\pi}{2}s)(-1)^{k/2}}\{I_k(\Delta,t;s)+I_k(\Delta,-t;s)\}\\
	= \,& \frac{2^{k+s-1}m^{k/2}f^{-s}}{\pi^{(s-1)/2}\Gamma(\frac{1-s}{2})\times \pi\Gamma(\frac{1-s}{2})^{-1}\Gamma(\frac{1+s}{2})^{-1}(-1)^{k/2}}\{I_k(\Delta,t;s)+I_k(\Delta,-t;s)\}\\
	=\, & 2^{k+s-1}m^{k/2}f^{-s}\pi^{(-s-1)/2}\Gamma(\tfrac{1+s}{2})(-1)^{k/2}\{I_k(\Delta,t;s)+I_k(\Delta,-t;s)\}.
\end{align*}
Here we use the two relations
\begin{align*}
	I_{k,1-s}(x) = I_{k,s}(x)=\frac{2^{1-k}\Gamma(1/2)}{\Gamma(k-1/2)}\Gamma(k-1+s)\Gamma(k-s)(x^2-1)^{-k/2+1/2}\mathfrak{P}_{-s}^{1-k}(x)
\end{align*}
for $x \in \CC-(-\infty,1]$ and
\begin{align*}
	I(f^2,t;s)+I(f^2,-t;s) = & \left(\frac{\Delta}{4}\right)^{(s-k)/2}2\cos\left(\frac{\pi}{2}s\right)(-1)^{k/2}
	\frac{\Gamma(k-1/2)\Gamma(1/2)}{\Gamma(k)}I_{k,s}\left(\frac{|t|}{\sqrt{\Delta}}\right)
\end{align*}
when $\Delta=f^2$ by \cite[p.134]{Zagier}, where $I_{k,s}(x)$ is the integral \cite[(56)]{Zagier}.
As the $L$-function associated with $\Delta=f^2$ has the expression
$L(s, \Delta)=\zeta(s)\sum_{d|f}\mu(d)\chi_1(d)d^{-s}\s_{1-2s}(\frac{f}{d})$
by \cite[Proposition 3 iii)]{Zagier}, we obtain
\begin{align*}
	& (-1)^{\#S}\JJ_{\rm hyp}^0(\ZZ|\a,z) \\
	=\, & \frac{1}{2}\zeta(s)\sum_{\substack{t \in \ZZ\\ t^2-4m=\Delta=f^2(\exists f\ge 1)}}m^{-1/2}f^s B_\Delta(s)\\
	&\times
	2^{k+s-1}m^{k/2}f^{-s}\pi^{(-s-1)/2}\Gamma(\tfrac{s+1}{2})(-1)^{k/2}\{I_k(\Delta,t;s)+I_k(\Delta,-t;s)\}\\
	=\, & m^{(k-1)/2}2^{k+s-1}\pi^{(-s-1)/2}\Gamma(\tfrac{s+1}{2})(-1)^{k/2} \times \frac{1}{2}\sum_{\substack{t \in \ZZ\\ \Delta:=t^2-4m=\square}}\{I_k(\Delta,t;s)+I_k(\Delta,-t;s)\}\zeta(s)B_\Delta(s)\\
	=\, & m^{(k-1)/2}\pi^{(-s-1)/2}\Gamma(\tfrac{s+1}{2})(-1)^{k/2}2^{k+s-2} \\
	&\times \sum_{\substack{t \in \ZZ\\ t^2-4m=\square}}\{I_k(t^2-4m,t;s)+I_k(t^2-4m,-t;s)\}L(s,t^2-4m),
\end{align*}
and hence
	\begin{align}\label{hyperbolic of Zagier}
		& \frac{(-1)^{k/2}(-1)^{\#S}m^{\frac{k-1}{2}}}{2^{s+k-2}\pi^{-s/2-1/2}\Gamma(\frac{s+1}{2})}\JJ_{\rm hyp}^0(\ZZ|\a,z) \\
		=\, &
		m^{k-1} \sum_{\substack{t \in \ZZ\\ t^2-4m=\square}}\{I_k(t^2-4m,t;s)+I_k(t^2-4m,-t;s)\}L(s,t^2-4m).\notag
	\end{align}

\subsection{Elliptic terms of Zagier's formulas}
\label{Elliptic terms of Zagier's formulas}
We have
\[\JJ_{\rm ell}^0(\ZZ|\a,z)=\frac{1}{2}\sum_{(t:n)_\QQ}\nr({\mathfrak d}_\Delta)^{\frac{z+1}{4}}L\left(\tfrac{z+1}{2},\varepsilon_\Delta\right)
\bfB_\ZZ^{(z)}(\a|\Delta;n\ff_\Delta^{-2})\,\Ocal_\infty^{\sgn(\Delta),(z)}\left(\tfrac{t}{\sqrt{|\Delta|}}\right)\]
with $\Delta:=t^2-4n$.
Note that $\Gamma_\RR(s+1)\Ocal_\infty^{-,(z)}(\tfrac{t}{\sqrt{|\Delta|}})$ is equal to
\begin{align*}
&\Gamma_\RR(s+1)\times \frac{1}{2}\times 2^k (1+\tfrac{t^2}{-\Delta})^{k/2}(-1)^{k/2}\left\{I_k(-4,\tfrac{2t}{\sqrt{|\Delta|}};s)+I_k(-4,\tfrac{-2t}{\sqrt{|\Delta|}};s)\right\}\\
	=\, & \Gamma_\RR(s+1)\times \frac{1}{2}\times 2^k (\tfrac{4m}{|\Delta|})^{k/2}(-1)^{k/2}\times 2^{s-k}|\Delta|^{(k-s)/2}
	\{I_k(\Delta,t;s)+I_k(\Delta,-t;s)\} \\
	=\, & m^{k/2}2^{s+k-1}|\Delta|^{-s/2} \pi^{(-s-1)/2}\Gamma(\tfrac{s+1}{2})(-1)^{k/2}\{I_k(\Delta,t;s)+I_k(\Delta,-t;s)\}
\end{align*}
for $\Delta<0$, which is the same as the case for $\Delta>0$ (See Appendix \ref{Hyperbolic terms of Zagier's formula}).
In the summation, $(t, n)_\QQ$ varies as in \S \ref{Trace formulas}.
However, we can restrict the variable $n$ of $(t,n)$ to the case ``$n=m$ and $t\in \ZZ$''.
Indeed, if $\a_q(s_q)=X_{\ord_q(m)}(q^{-s_q/2}+q^{s_q/2})$ for each $q \in S=S(m\ZZ)$ and
\begin{align*}
	\prod_{q \in S}\hat{\Scal}_q^{t^2-4n, (z)}(\a_q,\tfrac{n}{f^2})\neq 0,
\end{align*}
then, there must exist $c \in \ZZ(S)^\times \cap \NN$ such that $ct \in \ZZ$ and $nc^2=m$
by \cite[Lemma 2.7]{SugiyamaTsuzuki2019}.
As a result, $(t:n)_\QQ$ in the summation of $\JJ_{\rm ell}^{0}(\ZZ|\a,z)$ can be replaced with $(t:m)_\QQ$ such that $t \in \ZZ$.
Therefore, by the formula
\begin{align*}
	(-1)^{\#S} \bfB_{\ZZ}^{(z)}(\a|\Delta; \tfrac{n}{f^2})=m^{-1/2}|f|^{s}B_{\Delta}(s)
\end{align*}
proved in the same way as in Appendix \ref{Hyperbolic terms of Zagier's formula},
the elliptic term is transformed into
\begin{align*}
	& (-1)^{\#S}\JJ_{\rm ell}^0(\ZZ|\a,z) \\
	=\, & \frac{1}{2}\sum_{(t:n)_\QQ}|D|^{s/2}L_{\fin}(s, \varepsilon_\Delta)m^{-1/2}f^{s}B_{\Delta}(s)\\
	& \times
	m^{k/2}2^{s+k-1}|\Delta|^{-s/2} \pi^{(-s-1)/2}\Gamma(\tfrac{s+1}{2})(-1)^{k/2}\{I_k(\Delta,t;s)+I_k(\Delta,-t;s)\}\\
	=\, & m^{(k-1)/2}2^{s+k-2}\pi^{(-s-1)/2}\Gamma(\tfrac{s+1}{2})(-1)^{k/2}
	\hspace{-6mm}\sum_{\substack{t \in \ZZ\\ \Delta:=t^2-4m=Df^2\neq \square}}\hspace{-5mm}\{I_k(\Delta,t;s)+I_k(\Delta,-t;s)\}L_{\fin}(s, \varepsilon_\Delta)B_{\Delta}(s),
\end{align*}
where $L_\fin(s,\varepsilon_{\Delta})$ is the finite part of $L(s,\varepsilon_{\Delta})$. Consequently we obtain
	\begin{align}\label{elliptic of Zagier}
		& \frac{(-1)^{k/2}(-1)^{\#S}m^{\frac{k-1}{2}}}{2^{s+k-2}\pi^{-s/2-1/2}\Gamma(\frac{s+1}{2})}\JJ_{\rm ell}^0(\ZZ|\a,z) \\
		= &
		m^{k-1} \sum_{\substack{t \in \ZZ\\ t^2-4m\neq \square}}\{I_k(t^2-4m,t;s)+I_k(t^2-4m,-t;s)\}L(s, t^2-4m).\notag 
	\end{align}

\subsection{Comparison with Zagier's formula}
Combining \eqref{cuspidal of Zagier}, \eqref{unipotent of Zagier}, \eqref{hyperbolic of Zagier} and \eqref{elliptic of Zagier},
the quantity $c_m(s)$ is evaluated as
\begin{align*}
	c_m(s)= \, & \frac{(-1)^{k/2}(-1)^{\#S}m^{\frac{k-1}{2}}}{2^{s+k-2}\pi^{-s/2-1/2}\Gamma(\frac{s+1}{2})}
	\times(-1)^{\#S}C(k,\ZZ)\II_{\rm cusp}^0(\ZZ|\a,z) \\
	= & \frac{(-1)^{k/2}(-1)^{\#S}m^{\frac{k-1}{2}}}{2^{s+k-2}\pi^{-s/2-1/2}\Gamma(\frac{s+1}{2})}
	\{\JJ_{\rm unip}^0(\ZZ|\a,z)+\JJ_{\rm unip}^0(\ZZ|\a,-z)+\JJ_{\rm hyp}^0(\ZZ|\a,z)+\JJ_{\rm ell}^0(\ZZ|\a,z)\} \\
	=\, &
	\delta(m=\square)m^{\frac{k-1}{2}-\frac{s}{2}}
	\frac{(-1)^{k/2}\Gamma(k+s-1)\zeta(2s)}{2^{2s+k-3}\pi^{s-1}\Gamma(k)} \\
	&+2\, \delta(m = \square)m^{\frac{k-1}{2}+\frac{s-1}{2}} \frac{(-1)^{k/2}}{\Gamma(\frac{s+1}{2})\Gamma(\frac{1-s}{2})}
	\frac{\Gamma(s-1/2)\Gamma(k-s)}{\Gamma(k)}\zeta(2s-1)2^{s-k+1}\pi^{3/2} \\
	&+m^{k-1} \sum_{\substack{t \in \ZZ\\ t^2-4m=\square}}\{I_k(t^2-4m,t;s)+I_k(t^2-4m,-t;s)\}L(s,t^2-4m)\\
	&+m^{k-1} \sum_{\substack{t \in \ZZ\\ t^2-4m\neq \square}}\{I_k(t^2-4m,t;s)+I_k(t^2-4m,-t;s)\}L(s,t^2-4m)\\
	=\, & \delta(m=\square)m^{\frac{k-1}{2}-\frac{s}{2}}
	\frac{(-1)^{k/2}\Gamma(k+s-1)\zeta(2s)}{2^{2s+k-3}\pi^{s-1}\Gamma(k)}\\
	&  +m^{k-1} \sum_{t \in \ZZ}\{I_k(t^2-4m,t;s)+I_k(t^2-4m,-t;s)\}L(s,t^2-4m).
\end{align*}
Therefore the ST trace formula recovers Zagier's formula (Theorem \ref{Zagier formula}).
\begin{rem}{\rm
We can check that Zagier's formula is equivalent to the formula \cite[(1.15)]{SugiyamaTsuzuki2019} by the same computation in Appendix \ref{Comparison of the ST trace formula with Zagier's formula}.}
\end{rem}

\section{Comparison of the ST trace formula with Takase's formula}
\label{Comparison of the ST trace formula with Takase's formula}
Takase \cite{Takase} generalized Zagier's formula to the case of Hilbert cusp forms over a totally real number field $F$ whose narrow class number $h_F^+$ is $1$ in the same method as Zagier's \cite{Zagier}.
He considered the space $S_{k}(\fn,\omega)$ of Hilbert cusp forms of weight $k=(k_v)_{v \in \Sigma_\infty}$, level $\Gamma_0(\gn)$ and nebentypus $\omega$, and assumed $\min_{v\in\Sigma_\infty}k_v \ge 3$
and $\fn=\gf_\omega$, where $\gf_\omega$ is the conductor of $\omega$.
In Appendix \ref{Comparison of the ST trace formula with Takase's formula}, we prove that the ST trace formula recovers Takase's formula.

Let $F$ be a totally real number field of degree $d_F=[F:\QQ]<\infty$ and
$\go^\times_+$ the set of all totally positive units in $\go^\times$.
Let ${\rm Cl}_F$ and ${\rm Cl}^{+}_F$ be the ideal class group of $F$ and
the narrow ideal class group of $F$, respectively.
The cardinalities of ${\rm Cl}_F$ and of ${\rm Cl}^{+}_F$
are called the class number of $F$ and the narrow class number of $F$, respectively.
In this subsection, we assume $h_F^+=1$ as Takase did in \cite{Takase}.
Invoking the exact sequence
\[1\rightarrow \go^\times/\go_{+}^{\times}\rightarrow \prod_{v \in \Sigma_\infty}F_v^\times/F_{v,+}^\times \rightarrow {\rm Cl}_F^+\rightarrow {\rm Cl}_F \rightarrow 1,
\]
where $F_{v,+}^\times$ denotes the set of positive elements of $F_v^\times$,
we have the equality $h_F^+=h_F \times 2^{d_F}/ \#(\go^\times/\go_{+}^{\times})$.
Thus the assumption $h_F^+=1$ leads us $h_F=1$, $\#(\go^\times/\go_{+}^{\times})=2^{d_F}$ and $(\go^\times)^2=\go^\times_+$.

Let $\omega$ be a unitary character of $F^\times\bsl\AA^\times$ of the form
\[\omega(x)=\prod_{v \in \Sigma_\infty}|x_v|_v^{s_v}\sgn_v(x_v)^{e_v} \prod_{v \in \Sigma_\fin}|x_v|_v^{s_v}\lambda_v(\tilde{x_v}),\]
where $s_v \in i\RR$, $e_v \in \{0,1\}$, and $\l_v$ is a character of $\go_v^\times$. The element $\tilde{x_v}\in \go_v^\times$ is the projection of $x_v \in F_v^\times$.
Let $\chi_\omega$ be the character of the ideal group associated with $\omega$ such that $\chi_\omega(\gp) =\nr(\fp)^{-s_\fp}$ for prime ideals $\fp$ prime to the conductor $\ff_\omega$ of $\omega$.

Let $\gn$ be an integral ideal such that $\gf_\omega | \gn$.
Let $k =(k_v)_{v \in \Sigma_\infty}$ be an element of $\NN^{\Sigma_{\infty}}$ such that $k_v\ge 2$ and $k_v \equiv e_v \pmod 2$ for all $v \in \Sigma_\infty$.
Let $\bfK_0(\gn)$ be the Hecke congruence subgroup of level $\gn$ (cf.\ \cite[\S2.2]{SugiyamaTsuzukiActa}).
By $h_F^+=1$, we have $\GL_2(\AA)=\GL_2(F)(G_\infty^+ \times \bfK_0(\gn))$,
where
$G_\infty^+$ denotes the subgroup of $\GL_2(F\otimes \RR)$ consisting of elements with totally positive determinant.
Put
\[\Gamma_0(\gn)=\GL_2(F)\cap (G_\infty^+\times \bfK_0(\gn))=\{\gamma=[\begin{smallmatrix}
a & b \\ c&d
\end{smallmatrix}]\in \GL_2(\go) \ | \ c\in \fn, \det (\gamma)\gg0\},\]
where $\gg 0$ means to be totally positive (cf.\ \cite[p.140]{Takase}).
Let $S_{k}(\gn, \omega)$ be the space of ($\CC$-valued) holomorphic Hilbert cusp forms of weight $k$, level $\Gamma_0(\gn)$ and nebentypus $\omega$.
For $f \in S_{k}(\gn, \omega)$, $C^*(f,\ga)$ denotes the
Fourier coefficient of $f$ at an ideal $\ga$ of $\go$ modified as in \cite[p.142]{Takase}.
We consider the second $L$-function attached to $f \in S_k(\gn,\omega)$:
\[
L_2(s,f,\overline{\chi_\omega})=\zeta_{F,\fin}(2s)_\fn \sum_{(\fb,\fn)=1}\frac{\overline{\chi_\omega}(\fb) C^*(f,\fb^2)}{\nr(\fb)^{s+1}}, \quad \Re(s)>1,\]
where $\fb$ varies in the set of ideals of $\go$ relatively prime to $\gn$
and we set $\zeta_{F,\fin}(s)_\fn = \prod_{v \Sigma_\fin-S(\gn)}(1-q_v^{-s})^{-1}$.
The second $L$-function attached to $f$ is the symmetric square of the $L$-function of $f$ twisted by $\overline{\chi_\omega}$ if $f$ is a normalized Hecke eigenform.

Put
\begin{align}\label{C_a(s)}
	C_\fa(s)=\sum_{f \in B_k(\gn,\omega)}\frac{L_2(s,f,\overline{\chi_\omega}) }{(f,f)}C^*(f,\fa),
\end{align}
where $B_k(\gn,\omega)$ is an orthogonal basis of $S_k(\gn,\omega)$ consisting of normalized Hecke eigenforms, and $(f, f)$ is the square of the Petersson norm of $f$
(cf.\ \cite[p.140]{Takase}).
Note that the quantity $C_\ga(s)$ is the same as $C^*(\Phi_s,\ga)$ in \cite[p.144]{Takase}.
For $n,m\in \go$, we set
\[\Ical_k(\Delta,n;s)=\prod_{v \in \Sigma_\infty}\frac{\Gamma(k_v)}{\Gamma(k_v-1/2)\Gamma(1/2)}I_{k_v}(\Delta_v, n_v, s),\]
where $\Delta=n^2-4m$ and $I_{k_v}(\Delta_v,n_v,s)$ is the integral as in \cite[p.110]{Zagier}
(cf.\ \cite[p.155]{Takase}), and for $a \in F$ and a place $v$ of $F$, we denote by $a_v$ the image of $a$ under the embedding $F \hookrightarrow F_v$.
For
$n, m \in \go$ such that $m\gg0$ and $(\fn,m)=1$, set
\[C_{n,m}(\omega, \fb)=\{\prod_{v \in \Sigma_\infty}m_v^{(k_v-s_v)/2}\}\sum_{\substack{t \in \go/\fb \\ t^2+nt+m\equiv 0\pmod \gb}}\overline{\lambda}_\fn(t)\]
for an ideal $\fb \subset \go$ divided by $\gn$, where $\l_\fn(t)=\delta(S(t\go)\cap S(\gn)=\emptyset)\prod_{v \in S(\fn)}\l_v(t)$.
Further we set
\[L_\fn(n,m,\omega;s)=\frac{\zeta_{F,\fin}(2s)_\fn}{\zeta_{F,\fin}(s)} \sum_{\fn | \fb}\frac{C_{n,m}(\omega,\fb)}{\nr(\fb)^s}, \qquad \Re(s)\gg 1\]
(cf.\ \cite[p.159]{Takase}).
This equals the product of an $L$-function of $F$ and a finite character sum (cf.\ \cite[Proposition 3]{Takase}).
Put $\underline{k}=\min_{v\in\Sigma_\infty}k_v$.
Under the preparation above,  Takase's formula \cite[Proposition 2]{Takase}
is stated as follows.

\begin{thm}
	[Takase's formula \cite{Takase}] \label{Takase formula}
	Suppose $h_F^+=1$, $\underline{k}=\min_{v \in \Sigma_\infty}k_v>2$ and $\fn=\gf_\omega$.
	Let $\fa=(m)$ be an integral ideal such that $m$ is a totally positive generator of $\fa$.
	For $s\in \CC$ such that $2-\underline{k}<\Re(s)<\underline{k}-1$, i.e.,
	$|\Re(2s-1)|<\underline{k}-3$, we have
		\begin{align*}
		C_\fa(s)= &\bigg\{\prod_{v \in \Sigma_\infty}\frac{(4\pi)^{k_v-1}}{\Gamma(k_v-1)}\bigg\} D_F^{1/2} \zeta_{F,\fin}(2s)_\fn \delta(\ga=\square)\nr(\fa)^{(1-s)/2}\chi_\omega(\fa)^{1/2}\\
		& + \frac{1}{2}\{\prod_{v \in \Sigma_\infty}(-2\sqrt{-1})^{k_v}(4 \pi)^{s+k_v-2}\frac{\Gamma(1/2)\Gamma(k_v-1/2)}{\Gamma(s+k_v-1)\Gamma(k_v-1)}\}
		D_F^{1/2-s} \\
		& \quad \times \sum_{n \in \go}\sum_{\varepsilon \in \go^\times/\go_{+}^{\times}}\{\prod_{v \in \Sigma_\infty}\varepsilon_v^{k_v}\}\Ical_k(\varepsilon^2(n^2-4m), \varepsilon n; s)L_{\gn}(n,m,\omega;s).
	\end{align*}
\end{thm}

In order to recover Takase's formula from the ST trace formula,
we add the assumptions that $2\in \QQ$ is completely splitting in $F$ and that $\omega=\bf1$ is satisfied, throughout Appendix \ref{Comparison of the ST trace formula with Takase's formula}.
Further we restrict the range of $s$ to $|\Re(2s-1)|<\underline{k}-3$ to use the ST trace formula.
Then, $\fn$ satisfies $\fn=\gf_\omega=\go$.

\subsection{Cuspidal terms of Takase's formula}
Let us transform the left-hand side of Takase's formula.
Define $z \in \CC$ with $|\Re(z)|<\underline{k}-3$ by $s=\frac{z+1}{2}$.
For $\ga=\prod_{v \in S}(\gp_v\cap\go)^{n_v}$ with $n_v\ge1$, set $\a(\bfs) = \otimes_{v \in S}X_{n_v}(q_v^{-s_v/2}+q_v^{s_v/2})$.
For $f \in B(k,\go)$, we denote by $\pi_f$ the cuspidal automorphic representation $\pi_f$
associated with $f$.
Then, we have
$L_2(s,f, \overline{\chi_{\bf1}}) = L(s,\Sym^2(\pi_f))$
by $\omega=\bf1$ and $\gn=\go$. Hence, by the relation
\begin{align}\label{norm of new form}L(1,\Sym^2(\pi_f)) = \{\prod_{v \in \Sigma_\infty}\Gamma_\RR(2)\Gamma_\CC(k_v)\}L_{\fin}(1, \Sym^2(\pi_f))= 2^{-1}D_F^{-1}\{\prod_{v \in \Sigma_\infty}2^{k_v+1}\}(f,f)
\end{align}
from \cite[Proposition 1]{Takase}, the quantity $C_\ga(s)$ is described as
\begin{align*}
	C_\fa(s) = & \sum_{f\in B_k(\go,\bf1)}\frac{L_\fin(s, \Sym^2(f))}{(f,f)}C^*(f,\fa) \\
	= & 2^{-1}D_F^{-1}\bigg\{\prod_{v \in \Sigma_\infty}\frac{2^{k_v+1}}{\Gamma_\RR(s+1)\Gamma_{\CC}(s+k_v-1)} \bigg\}\nr(\fa)^{1/2}\sum_{\pi \in \Pi_{\rm cus}(k,\go)}\frac{L(s,\Sym^2(\pi))}{L(1,\Sym^2(\pi))}\a(\nu_S(\pi)).
\end{align*}
From this and the ST trace formula, the left-hand side of Takase's formula is transformed as
	\begin{align}\label{cuspidal of Takase}
		C_{\fa}(s) = & 2D_F^{1/2-z}(-1)^{\#S}C(k,\go)^{-1} 2^{-1}D_F^{-1}\bigg\{\prod_{v \in \Sigma_\infty}\frac{2^{k_v+1}}{\Gamma_\RR(s+1)\Gamma_{\CC}(s+k_v-1)} \bigg\}\nr(\fa)^{1/2} \\
		& \times (-1)^{\#S}C(k,\go)\II_{\rm cusp}^0(\go|\a,z)\notag \\
		= & D_F^{1/2-z}(-1)^{\#S}\bigg\{\prod_{v \in \Sigma_{\infty}}\frac{k_v-1}{4\pi}\frac{2^{k_v+1}}{\Gamma_\RR(s+1)\Gamma_{\CC}(s+k_v-1)} \bigg\}\nr(\fa)^{1/2} \notag \\
		& \times\bigg[D_F^{z/4}\{\JJ_{\rm unip}^0(\go|\a,z)+\JJ_{\rm unip}^0(\go|\a,-z)\} +\JJ_{\rm hyp}^0(\go|\a,z) +\JJ_{\rm ell}^0(\go|\a,z)\bigg].\notag 
	\end{align}

\subsection{Unipotent terms of Takase's formula}
Recall the expression
\begin{align*}
	\JJ_{\rm unip}^0(\go|\a,z) = & 
	D_F^{-\frac{z}{4}}D_{F}^{z}\zeta_{F,\fin}(2s)\delta(\ga=\square)(-1)^{\#S}\nr(\ga)^{-s/2}\prod_{v \in \Sigma_\infty}\Gamma_\RR(2s)2^{2-2s}\pi^{1-s/2}\frac{\Gamma(k_v+s-1)}{\Gamma(s/2)\Gamma(k_v)}
\end{align*}
with $s=\frac{z+1}{2}$ (cf.\ Lemma \ref{explicit unip int}).
Thus we have the formula
\begin{align}\label{unipotent z of Takase}
	& D_F^{1/2-z}(-1)^{\#S}\bigg\{\prod_{v \in \Sigma_{\infty}}\frac{k_v-1}{4\pi}\frac{2^{k_v+1}}{\Gamma_\RR(s+1)\Gamma_{\CC}(s+k_v-1)} \bigg\}\nr(\fa)^{1/2}\times D_F^{z/4} \JJ_{\rm unip}^0(\go|\a,z)\\
	= & D_F^{1/2}\zeta_{F,\fin}(2s)\nr(\ga)^{\frac{1-s}{2}}\delta(\ga=\square) \prod_{v \in \Sigma_\infty}\frac{(4\pi)^{k_v-1}}{\Gamma(k_v-1)}.\notag
\end{align}
This coincides with the first term of the right-hand side of Takase's formula.
In a similar way, we have
\begin{align*}
	\JJ_{\rm unip}^0(\go|\a,-z) = &  
	D_F^{\frac{z-2}{4}}\zeta_{F,\fin}(2s-1)\delta(\ga=\square)(-1)^{\#S}\nr(\ga)^{\frac{s-1}{2}}\prod_{v \in \Sigma_\infty}\Gamma_\RR(2s-1) 2^{2s}\pi^{\frac{s+1}{2}}\frac{\Gamma(k_v-s)}{\Gamma(\frac{1-s}{2})\Gamma(k_v)}
\end{align*}
and whence
\begin{align}\label{unipotent -z of Takase}
	& D_F^{1/2-z}(-1)^{\#S}\bigg\{\prod_{v \in \Sigma_{\infty}}\frac{k_v-1}{4\pi}\frac{2^{k_v+1}}{\Gamma_\RR(s+1)\Gamma_{\CC}(s+k_v-1)} \bigg\}\nr(\fa)^{1/2}\times D_F^{z/4} \JJ_{\rm unip}^0(\go|\a,-z)\\
	= & D_F^{-z/2}\zeta_{F,\fin}(2s-1)\times 
	\delta(\ga=\square)\nr(\ga)^{s/2} \notag \\
	&\times \prod_{v \in \Sigma_\infty}\frac{k_v-1}{4\pi}\frac{2^{k_v+1}}{\Gamma_\RR(s+1)\Gamma_{\CC}(s+k_v-1)} \Gamma_\RR(2s-1)2^{2s}\pi^{\frac{s+1}{2}}\frac{\Gamma(k_v-s)}{\Gamma(\frac{1-s}{2})\Gamma(k_v)} \notag \\
	=& D_F^{1/2-s}\zeta_{F,\fin}(2s-1)\delta(\ga=\square)\nr(\ga)^{s/2}
	\prod_{v \in \Sigma_\infty}\frac{\Gamma(1/2)\Gamma(k_v-s)\Gamma(s-1/2)}{\Gamma(k_v-1)\Gamma(s+k_v-1)}\cos\left(\frac{\pi}{2}s\right)(4\pi)^{s+k_v-2}2^{s+1}. \notag
\end{align}
In Takase's formula, we note the formula
\begin{align*}
	\sum_{\varepsilon \in \go^\times/\go^\times_+}\{\prod_{v \in \Sigma_\infty}\varepsilon_v^{k_v}\}\Ical_k(0,\varepsilon n;s)=\prod_{v \in \Sigma_\infty}\frac{\Gamma(s-1/2)\Gamma(k_v-s)}{\Gamma(k_v-1/2)}|n|_v^{s-k_v}(-1)^{k_v/2}2\cos\left(\frac{\pi}{2}s\right),
\end{align*}
which is proved by \cite[(13)]{Zagier} and the surjectivity of the mapping $\sgn : \go^\times/\go^{\times}_+\rightarrow\{\pm1\}^{\Sigma_{\infty}}$ under $h_F^+=1$.
Thus, the term for $n^2-4m=0$ of Takase's formula is described as
\begin{align*}
	& \frac{1}{2}\bigg\{\prod_{v \in \Sigma_\infty}2^{k_v}(-1)^{k_v/2}(4\pi)^{s+k_v-2}\frac{\Gamma(1/2)\Gamma(k_v-1/2)}{\Gamma(s+k_v-1)\Gamma(k_v-1)}\bigg\}D_F^{1/2-s}\\
	& \times \sum_{\substack{n \in \go\\ n^2-4m=0}}\sum_{\varepsilon \in \go^\times/\go^\times_+} \{\prod_{v \in \Sigma_\infty}\varepsilon_v^{k_v}\}\Ical_{k}(0,\varepsilon n ;s)L_{\go}(n,m,{\bf 1};s) \\
	= & 2\times \frac{1}{2}\bigg\{\prod_{v \in \Sigma_\infty}2^{k_v}(-1)^{k_v/2}(4\pi)^{s+k_v-2}\frac{\Gamma(1/2)\Gamma(k_v-1/2)}{\Gamma(s+k_v-1)\Gamma(k_v-1)}\\
	& \times\frac{\Gamma(s-1/2)\Gamma(k_v-s)}{\Gamma(k_v-1/2)}(-1)^{k_v/2}2\cos\left(\frac{\pi}{2}s\right)\bigg\}\\
	& \times \{\prod_{v \in \Sigma_\infty}|n_v|_v^{s-k_v} \}\times D_F^{1/2-s} \delta(\ga=\square) \zeta_{F,\fin}(2s-1)\{\prod_{v \in \Sigma_\infty}2^{-k_v}n_v^{k_v}\}\\
	= & 2\times \frac{1}{2}\bigg\{\prod_{v \in \Sigma_\infty}2^{s+1}(4\pi)^{s+k_v-2}\frac{\Gamma(1/2)\Gamma(s-1/2)\Gamma(k_v-s)}{\Gamma(s+k_v-1)\Gamma(k_v-1)}
	\cos\left(\frac{\pi}{2}s\right)\bigg\}\\
	& \times \nr(\fa)^{s/2}\times D_F^{1/2-s} \delta(\ga=\square) \zeta_{F,\fin}(2s-1).
\end{align*}
This exactly coincides with the unipotent term for $-z$ of the ST trace formula.

\subsection{Hyperbolic terms of Takase's formula}
\label{Hyperbolic terms of Takase's formula}
For $t,m \in \go$, the equality $t^2-4m=Df^2$ means
that $f$ is an integer of $F$ and $D$ generates the relative discriminant of the quadratic extension
$F(\sqrt{t^2-4n})/F$. Such a $D$ with $D\neq 1$ is called a fundamental discriminant over $F$.
For a fundamental discriminant $D$ over $F$, let $\chi_D$ denotes the character
over the ideals of $F$
corresponding to $F(\sqrt{D})/F$ by class field theory.
We set $\chi_D=\bf1$ for $D=1$.

Recall that $m$ is now taken to be totally positive such that $\ga=(m)$.
By abuse of notation, $\gp_v$ for $v \in \Sigma_\fin$ is regarded as a prime ideal $\gp_v\cap\go$ of $\go$.

Here is a lemma seen as \cite[Proposition 3]{Takase}.
\begin{lem}For $n \in \go$, let $f$ be an element of $\go$ such that $n^2-4m=Df^2$
	with $D$ being a fundamental discriminant over $F$.
	When $f=0$, we have
	\begin{align*}
		L_\fn(n,m, \omega; s) =\{\prod_{v \in \Sigma_\infty}2^{-k_v}n_v^{k_v}\}\chi_\omega(n/2)\zeta_{F,\fin}(2s-1)\nr(\fn)^{1-s}\prod_{v \in S(\fn)}(1+q_v^{s-1})(1-q_v^{-s}).
	\end{align*}
Here we put $D=1$ for $f=0$.
In particular, if $\omega=\bf 1$ and $\gn=\go$, then we have
	\begin{align*}
		L_\go(n,m, {\bf1}; s) =  \{\prod_{v \in \Sigma_\infty}2^{-k_v}n_v^{k_v}\}\zeta_{F,\fin}(2s-1).
	\end{align*}

	When $f \neq0$, we have
	\begin{align*}
		L_\fn(n,m, \omega; s) =& L_{\fin}(s,\chi_D) \times \prod_{v \in S(\fn)}(1-q_v^{-s})\times \sum_{\fb | f_\fn}\nr(f_\fn\fb^{-1})^{1-2s}\prod_{v \in S(\fb)}(1-\chi_D(\fp_v)q_v^{-s})\\
		& \times \sum_{\fb|\fn(f)\fn^{-1}}C_{n,m}(\omega, \fn(f)\fb^{-1})\nr(\fn(f)\fb^{-1})^{-s}\prod_{v \in S(\fb)}(1-\chi_D(\fp_v)q_v^{-s}),
	\end{align*}
	where $L_{\fin}(s,\chi_D)$ is the Hecke $L$-function associated with $\chi_D$ without archimedean local $L$-factors, $f_\gn=\prod_{v \notdivide \gn}\gp_v^{\ord_{v}(f)}$,
	$\gn(f)=\prod_{v \in S(\gf)}(\gp_v)^{c(\gp_v)}$ with $c(\gp_v)=\max(2\ord_v(f)+1, \ord_v(\gn f))$.
	In particular, if $\omega=\bf 1$ and $\gn=\go$, then we have
	\begin{align*}
		L_\go(n,m, {\bf1}; s)
		= & L_\fin(s,\chi_D) \times \sum_{\fb | f}\nr(f\fb^{-1})^{1-2s}\prod_{v \in S(\fb)}(1-\chi_D(\fp_v)q_v^{-s}) \times C_{n,m}({\bf 1}, \fo)
	\end{align*}
	and \[C_{n,m}({\bf1}, \go) = \{\prod_{v \in \Sigma_\infty}m_v^{k_v/2}\}\, \#\{t\in \go/\go \ | \ t^2+nt+m \in \go\}=\prod_{v \in \Sigma_\infty}m_v^{k_v/2}.\]
\end{lem}

\begin{lem}For $\Delta\in\go$ with $\Delta=Df^2$, we have
	\begin{align*}B_\Delta^F(s) := \sum_{\gb|f }\nr(f\gb^{-1})^{1-2s}\prod_{v \in S(\gb)}(1-\chi_D(\gp_v)\nr(\gp_v)^{-s})=\sum_{\fc|f}\mu(\fc)\chi_D(\fc)\nr(\fc)^{-s}\s_{1-2s}(f{\mathfrak \fc}^{-1}),
	\end{align*}
where $\mu$ is the M\"obius function for $F$.
\end{lem}
\begin{proof}
	Let $\gb=\prod_{j=1}^g\gp_j^{\nu_j}$ be the prime ideal decomposition.
	The left-hand side equals
	\begin{align*}
		& \sum_{\gb|f}\nr(f\gb^{-1})^{1-2s}\sum_{\substack{l\le g\\ \ 1\le i_1<\cdots<i_l\le g}}(-1)^{l}\chi_{D}(\gp_{i_1})\cdots\chi_D(\gp_{i_l})\nr(\gp_{i_1}\cdots\gp_{i_l})^{-s} \\
		= & \sum_{\gb|f}\nr(f\gb^{-1})^{1-2s} \sum_{\fc|\fb}\mu(\fc)\chi_D(\fc)\nr(\fc)^{-s} 
		=  \sum_{\fc|f} \sum_{\substack{\fb|f\\ \fb=\fc\fm (\exists \fm)}}\nr(f\gb^{-1})^{1-2s}\mu(\fc)\chi_D(\fc)\nr(\fc)^{-s} \\
		= & \sum_{\fc|f}\mu(\fc)\chi_D(\fc)\nr(\fc)^{-s} \sum_{\fm|f\fc^{-1}}\nr(f\fc^{-1}\fm^{-1})^{1-2s}.
	\end{align*}
	This completes the proof.
\end{proof}

From these lemmas, we obtain
\begin{align*}
	& \frac{1}{2}\bigg(\prod_{v \in \Sigma_\infty}(-2\sqrt{-1})^{k_v}(4 \pi)^{s+k_v-2}\frac{\Gamma(1/2)\Gamma(k_v-1/2)}{\Gamma(s+k_v-1)\Gamma(k_v-1)}\bigg)
	D_F^{1/2-s} \\
	& \times \sum_{\substack{n \in \go\\ n^2-4m\neq0}}\sum_{\varepsilon \in \go^\times/\go_{+}^{\times}}
	\bigg\{\prod_{v \in \Sigma_\infty}\frac{\varepsilon_v^{k_v}\Gamma(k_v)}{\Gamma(k_v-1/2)\Gamma(1/2)}I_{k_v}(\varepsilon_v^2(n_v^2-4m_v), \varepsilon_vn_v, s)\bigg\}
	L_\gn(n,m,{\bf 1};s)\\
	= &
	\frac{1}{2}\bigg\{\prod_{v \in \Sigma_\infty}(2\sqrt{-1})^{k_v}(4 \pi)^{s+k_v-2}\frac{k_v-1}{\Gamma(s+k_v-1)}\bigg\}
	D_F^{1/2-s} \\
	& \times \sum_{\substack{n \in \go\\ n^2-4m=Df^2\neq0}}\sum_{\varepsilon \in \go^\times/\go_{+}^{\times}}
	\{\prod_{v \in \Sigma_\infty}\varepsilon_v^{k_v}I_{k_v}(\varepsilon_v^2(n_v^2-4m_v), \varepsilon_vn_v, s)\} L_{\fin}(s,\chi_D) \\
	& \times\sum_{\gb|f}\nr(f\gb^{-1})^{1-2s}\prod_{v \in S(\gb)}(1-\chi_D(\gp_v)q_v^{-s})\{\prod_{v \in \Sigma_\infty}m_v^{k_v/2}\}\\
	= &
	\frac{1}{2}\bigg\{\prod_{v \in \Sigma_\infty}m_v^{k_v/2}(2\sqrt{-1})^{k_v}(4 \pi)^{s+k_v-2}\frac{k_v-1}{\Gamma(s+k_v-1)}\bigg\}
	D_F^{1/2-s} \\
	& \times \sum_{\substack{n \in \go \\ \Delta:=n^2-4m \\
			\Delta=Df^2\neq0}}
	\bigg\{\prod_{v \in \Sigma_\infty}(I_{k_v}(n_v^2-4m_v, n_v, s)+I_{k_v}(n_v^2-4m_v, -n_v, s))\bigg\} \times L_{\fin}(s,\chi_D) B_\Delta^F(s).
\end{align*}
Here we use the fact that the mapping $\sgn : \go^\times/\go^{\times}_+\rightarrow\{\pm1\}^{\Sigma_{\infty}}$ is surjective under $h_F^+=1$.

We consider the hyperbolic term.
As for the ST trace formula,
we recall 
\begin{align*}\JJ_{\rm hyp}^0(\go|\a,z) 
	= & \frac{1}{2}D_F^{(z-1)/2}\zeta_{F,\fin}(\tfrac{1+z}{2})\sum_{a \in \go(S)_+^\times-\{1\}}
	\bfB_\go^{(z)}(\a|1; \tfrac{a}{(a-1)^2}\go)\,\prod_{v \in \Sigma_\infty}\Gamma_\RR(\tfrac{1+z}{2})\Ocal_v^{+,(z)}(\tfrac{a+1}{a-1}).
\end{align*}
Here we notice \begin{align*}
	\Gamma_\RR(\tfrac{z+1}{2})\Ocal_v^{+,(z)}(b) = & 2^{k_v+s-1}m_v^{k_v/2}f_v^{-s}\pi^{(-s-1)/2}\Gamma(\tfrac{1+s}{2})(-1)^{k_v/2}\{I_{k_v}(\Delta,t;s)+I_{k_v}(\Delta,-t;s)\}
\end{align*}
with $\Delta=t^2-4m$. In the same way as \eqref{identity of B} for $F=\QQ$, we obtain
\begin{align*}
	(-1)^{\#S}\bfB_{\go}(\a|1; a(a-1)^{-2}\go)=\nr(\ga)^{-1/2}\nr(f)^{s}B_\Delta^F(s)
\end{align*}
for $a=\frac{t+f}{t-f} \in \go(S)_{+}^{\times}-\{1\}$ such that $t^2-4m=f^2$.
Hence, we obtain
\begin{align*}
	\JJ_{\rm hyp}^0(\go|\a,z)= & \frac{1}{2}D_F^{(z-1)/2}\zeta_{F, \fin}(\tfrac{1+z}{2})\sum_{\substack{t \in \go \\ \Delta:=t^2-4m=f^2\neq 0}}
	(-1)^{\# S}\nr(\ga)^{-1/2}B_\Delta^F(s) \\
	& \times
	\,\big[\prod_{v \in \Sigma_\infty}2^{k_v+s-1}m_v^{k_v/2}\pi^{(-s-1)/2}\Gamma(\tfrac{1+s}{2})(-1)^{k_v/2}\{I_{k_v}(\Delta,t;s)+I_{k_v}(\Delta,-t;s)\}\big],
\end{align*}
and whence
\begin{align}\label{hyperbolic of Takase}
	& D_F^{1/2-z}(-1)^{\#S}\bigg\{\prod_{v \in \Sigma_{\infty}}\frac{k_v-1}{4\pi}\frac{2^{k_v+1}}{\Gamma_\RR(s+1)\Gamma_{\CC}(s+k_v-1)} \bigg\}\nr(\fa)^{1/2}\times \JJ_{\rm hyp}^0(\go|\a,z)\\
	= & \frac{1}{2}D_F^{1/2-s} \bigg\{\prod_{v \in \Sigma_{\infty}}m_v^{k_v/2}(-1)^{k_v/2}(4\pi)^{s+k_v-2}2^{k_v}\frac{k_v-1}{\Gamma(s+k_v-1)} \bigg\} \notag \\
	& \times \sum_{\substack{t \in \go \\ \Delta:=t^2-4m=f^2\neq0}}[\prod_{v \in \Sigma_{\infty}}\{I_{k_v}(t^2-4m, t, s)+I_{k_v}(t^2-4m, -t, s)\}]\zeta_{F, \fin}(s)B_{\Delta}^F(s).\notag
\end{align}
This coincides with the term for $n^2-4m=f^2 \neq 0$ of Takase's formula.

\subsection{Elliptic terms of Takase's formula and comparison}
We recall the condition $h_F^+=1$.
Then, the elliptic term is descirebed in the same way as the hyperbolic term:
\begin{align}\label{elliptic of Takase}
	& D_F^{1/2-z}(-1)^{\#S}\bigg\{\prod_{v \in \Sigma_{\infty}}\frac{k_v-1}{4\pi}\frac{2^{k_v+1}}{\Gamma_\RR(s+1)\Gamma_{\CC}(s+k_v-1)} \bigg\}\nr(\fa)^{1/2}\times \JJ_{\rm ell}^0(\go|\a,z)\\
	= & \frac{1}{2}D_F^{1/2-s} \bigg\{\prod_{v \in \Sigma_{\infty}}m_v^{k_v/2}(-1)^{k_v/2}(4\pi)^{s+k_v-2}2^{k_v}\frac{k_v-1}{\Gamma(s+k_v-1)} \bigg\} \notag \\
	& \times \sum_{\substack{t \in \go \\ \Delta:=t^2-4m \\
			\Delta =Df^2,\ D\neq1}}[\prod_{v \in \Sigma_{\infty}}\{I_{k_v}(t^2-4m, t, s)+I_{k_v}(t^2-4m, -t, s)\}]L_{\fin}(s, \varepsilon_\Delta)B_{\Delta}^F(s). \notag
\end{align}
This coincides with the term for $n^2-4m=Df^2\neq 0$ with $D\neq 1$ of Takase's formula.

From \eqref{cuspidal of Takase}, \eqref{unipotent z of Takase}, \eqref{unipotent -z of Takase},
\eqref{hyperbolic of Takase} and \eqref{elliptic of Takase},
the ST trace formula recovers Takase's formula for $\omega=\bf1$ and $\gn=\go$ (Theorem \ref{Takase formula}).

\section{Comparison of the ST trace formula with Mizumoto's formula}
\label{Comparison of the ST trace formula with Mizumoto's formula}

Mizumoto \cite{Mizumoto} gave a new proof of Zagier's formula using Poincar\'e series.
His method is valid for holomorphic Hilbert modular forms over a totally real number field $F$ of parallel weight $k$ and level
$\SL_2(\go)$ with trivial nebentypus.
In Appendix \ref{Comparison of the ST trace formula with Mizumoto's formula}, we prove that the ST trace formula recovers
Mizumoto's formula.

Let $F$ be a totally real number field with $d_F=[F:\QQ]<\infty$.
For $a \in F$ and a place $v$ of $F$, we denote by $a_v$ the image of $a$ under the embedding $F \hookrightarrow F_v$.
Assume that the narrow class number $h_F^+$ of $F$ equals one.
Then, the different of $F/\QQ$ has a totally positive generator $\mathfrak{d} \in \go$.
Mizumoto's formula \cite[(4.6) and (5.9)]{Mizumoto} is stated as follows.
\begin{thm}[Mizumoto's formula \cite{Mizumoto}]\label{Mizumoto formula}
	Let $k$ be an even integer such that $k\ge 4$.
	For $\fa = (m)$ with $m\gg0$ and $s \in \CC$ such that $2-k<\Re(s)<k-1$, i.e., $|\Re(2s-1)|<2k-3$, we have
	\begin{align*}
		&\left(\frac{\Gamma(k-1)}{(4\pi)^{k-1}}\right)^{d_F}D_F^{-1/2}\sum_{f \in \Bcal(k,\fo)}\frac{L_\fin(s, \Sym^2(f))}{(f,f)} \nr(\ga)^{k/2-1}C^*(f,\fa)\\
		= & \delta(\ga=\square)(-1)^{d_F k/2}D_F^{-s}\nr(\fa)^{(k-1)/2+(s-1)/2} \, (2\pi)^{d_F}2^{d_F(s-1)}
		  \bigg(\frac{\pi^{s-1/2}\Gamma(\frac{k-s}{2})\Gamma(s-1/2)}{\Gamma(\frac{k+s}{2})\Gamma(\frac{k+s-1}{2})\Gamma(\frac{1-k+s}{2})}\bigg)^{d_F}\\
		&\times \zeta_{F,\fin}(2s-1)
		 +\delta(\ga=\square)\nr(\fa)^{(k-1-s)/2}\zeta_{F,\fin}(2s) \\
		& +(-1)^{d_F k/2}D_F^{-1}\nr(\fa)^{(k-1)/2}2^{d_F-1}\pi^{d_F}\sum_{\substack{t \in \go\\ t^2-4m\neq 0}}\{\prod_{v \in \Sigma_\infty}I_v(t,m,s)\}L_F(s,t^2-4m),
	\end{align*}
	where
$C^*(f,\fa)$ is the modified Fourier coefficient used in Appendix \ref{Comparison of the ST trace formula with Takase's formula}, and
$I_v(t,m,s)$ is a meromorphic continuation to $\CC$
	of the integral
	\[I_v(t,m,s)=
	2\int_{0}^{\infty}\cos\left(\frac{2\pi t_v}{{\mathfrak d}_v}y\right)y^{-s}J_{k-1}\bigg(4\pi\frac{\sqrt{m_v}}{{\mathfrak d}_v}y\bigg)dy, \qquad 1/2<\Re(s)<k\]
as a function in $s$ by \cite[(2.20)]{Mizumoto}.
	The $L$-function $L_F(s,t^2-4m)$ is given by
	\begin{align*}
		L_F(s,t^2-4m)
		=\begin{cases}
			\zeta_{F,\fin}(2s-1)  & (t^2-4m=0),\\
			L_{\fin}(s, \chi_{D})B_\Delta^F(s) &  (t^2-4m=\Delta=Df^2\neq 0)
		\end{cases}
	\end{align*}
as in \cite[Proposition 1]{Mizumoto}. Here $D$ is taken as a fundamental discriminant over $F$
in the sense of Appendix \ref{Hyperbolic terms of Takase's formula}.
\end{thm}

%

\begin{rem}{\rm 
		In the left-hand side of Mizumoto's formula, $D_F^{k-1/2}$ is used instead of $D_F^{-1/2}$ in
		\cite[(5.9)]{Mizumoto}. However,
		we need a modification of Mizumoto's original formula as in Theorem \ref{Mizumoto formula} above.
		Indeed,
		for the Fourier coefficient $a(\ga)$ of a Hilbert cusp form $f$ for an integral ideal $\ga$
		in Mizumoto \cite{Mizumoto},
Takase's quantity $C^*(f, \ga)$
is described as $a(\ga)\nr(\ga)^{1-k/2}D_F^{k/2}$ but not $a(\ga)\nr(\ga)^{1-k/2}$.
	}
\end{rem}

For recovering Mizumoto's formula from the ST trace formula, assume that $2 \in \QQ$ is completely splitting in $F$.
Let us consider the case of $l=(k,k,\ldots,k)$, $\fn=\go$,
$\a(\bfs) = \otimes_{v \in S}X_{n_v}(q_v^{-s_v/2}+q_v^{s_v/2})$
for $\ga=\prod_{v \in S}(\gp_v\cap\go)^{n_v}$ with $n_v\ge1$. and $s=\tfrac{z+1}{2}$ with $|\Re(2s-1)|<k-3$ for using the ST trace formula.

The left-hand side of Mizumoto's formula is equal to $\{\frac{\Gamma(k-1)}{(4\pi)^{k-1}}\}^{d_F}D_F^{-1/2}\nr(\ga)^{k/2-1} \times C_{\ga}(s)$, where $C_{\ga}(s)$ the quantity introduced as \eqref{C_a(s)} in Takase's formula.
Hence, it is enough to consider
\begin{align}\label{cuspidal of Mizumoto}& \left(\frac{\Gamma(k-1)}{(4\pi)^{k-1}}\right)^{d_F}D_F^{-1/2}\nr(\ga)^{k/2-1}\\
	& \times D_F^{1/2-z}(-1)^{\#S}\bigg\{\prod_{v \in \Sigma_{\infty}}\frac{k-1}{4\pi}\frac{2^{k+1}}{\Gamma_\RR(s+1)\Gamma_{\CC}(s+k-1)} \bigg\}\nr(\fa)^{1/2} \times \II_{\rm cusp}^0(\go|\a_\ga,z)\notag
\end{align}
to recover Mizumoto's formula from the ST trace formula.

\subsection{Unipotent terms of Mizumoto's formula}
By \eqref{unipotent z of Takase}, the unipotent term for $z$ is described as
\begin{align}\label{unipotent z of Mizumoto}
	& \left(\frac{\Gamma(k-1)}{(4\pi)^{k-1}}\right)^{d_F}D_F^{-1/2}\nr(\ga)^{k/2-1} \\
	& \times D_F^{1/2-z}(-1)^{\#S}\bigg\{\prod_{v \in \Sigma_{\infty}}\frac{k-1}{4\pi}\frac{2^{k+1}}{\Gamma_\RR(s+1)\Gamma_{\CC}(s+k-1)} \bigg\}\nr(\fa)^{1/2}\times D_F^{z/4} \JJ_{\rm unip}^0(\go|\a,z) \notag \\
	= & \zeta_{F,\fin}(2s)\nr(\ga)^{\frac{k-1-s}{2}}\delta(\ga=\square).\notag
\end{align}
Similarly, by \eqref{unipotent -z of Takase}, the unipotent term for $-z$ is described as
\begin{align}\label{unipotent -z of Mizumoto}
	& \left(\frac{\Gamma(k-1)}{(4\pi)^{k-1}}\right)^{d_F}D_F^{-1/2}\nr(\ga)^{k/2-1} \\
	& \times D_F^{1/2-z}(-1)^{\#S}\bigg\{\prod_{v \in \Sigma_{\infty}}\frac{k-1}{4\pi}\frac{2^{k+1}}{\Gamma_\RR(s+1)\Gamma_{\CC}(s+k-1)} \bigg\}\nr(\fa)^{1/2}\times D_F^{z/4} \JJ_{\rm unip}^0(\go|\a,-z) \notag \\
	=& D_F^{-s}\zeta_{F,\fin}(2s-1)\delta(\ga=\square)\nr(\ga)^{\frac{s+k-2}{2}}
	\times \bigg\{\prod_{v \in \Sigma_\infty}\frac{\Gamma(k-s)\Gamma(s-1/2)}{\Gamma(s+k-1)}\cos\left(\frac{\pi}{2}s\right)2^{2s-1} 2^s \pi^{s-1/2}\bigg\}.\notag
\end{align}
By virtue of the duplication formula and the formula
$\cos(\frac{\pi}{2}s)=(-1)^{k/2}\pi\Gamma(\frac{1-s+k}{2})^{-1}\Gamma(\frac{1-k+s}{2})^{-1}$
induced from the reflection formula, the last line above is described as
\[
D_F^{-s}\zeta_{F,\fin}(2s-1)\delta(\ga=\square)\nr(\ga)^{\frac{s+k-2}{2}}
\bigg\{\prod_{v \in \Sigma_\infty}\pi^{s-1/2}\frac{\Gamma(\tfrac{k-s}{2})\Gamma(s-1/2)}{\Gamma(\tfrac{s+k-1}{2})\Gamma(\tfrac{s+k}{2})\Gamma(\tfrac{1-k+s}{2})}(-1)^{k/2}(2\pi)2^{s-1}\bigg\}.
\]
Consequently, the unipotent terms of the ST trace formula
coincide with the first two terms in the right-hand side
of Mizumoto's formula.

\subsection{Hyperbolic and elliptic terms of Mizumoto's formula, and comparison}

The sum of the hyperbolic and the elliptic terms
of the ST trace formula is described as
\begin{align}\label{hyperbolic elliptic of Mizumoto}
	& \left(\frac{\Gamma(k-1)}{(4\pi)^{k-1}}\right)^{d_F}D_F^{-1/2}\nr(\ga)^{k/2-1} \times D_F^{1/2-z}(-1)^{\#S}\bigg\{\prod_{v \in \Sigma_{\infty}}\frac{k-1}{4\pi}\frac{2^{k+1}}{\Gamma_\RR(s+1)\Gamma_{\CC}(s+k-1)} \bigg\} \\
	& \times \nr(\fa)^{1/2} \{\JJ_{\rm hyp}^0(\go|\a,z)+
	\JJ_{\rm ell}^0(\go|\a,z)\} \notag \\
	= & \frac{1}{2}D_F^{-s} \nr(\ga)^{k-1} (-1)^{d_F k/2} \bigg\{\prod_{v \in \Sigma_{\infty}}(4\pi)^{s-1}2^{k}\frac{\Gamma(k)}{\Gamma(s+k-1)} \bigg\} \notag \\
	& \times \sum_{\substack{t \in \go \\ \Delta:=t^2-4m\neq 0}}\prod_{v \in \Sigma_{\infty}}\{I_{k}(t^2-4m, t, s)+I_{k}(t^2-4m, -t, s)\}L_{\fin}(s, \varepsilon_\Delta)B_{\Delta}^F(s) \notag
\end{align}
by virtue of \eqref{hyperbolic of Takase} and \eqref{elliptic of Takase}.
In order to transform this into the series over $t \in \go$
appearing in Mizumoto's formula, we have only to check the following.
\begin{lem}\label{Takase I = Mizumoto I}
Suppose $t \in \go$ and $\Delta=t^2-4m \neq 0$. Take any $v \in \Sigma_\infty$.
For $s \in \CC$ with $2-k<\Re(s)<k-1$, we have
\[I_v(t,m,s) = \mathfrak{d}_v^{1-s}m_v^{(k-1)/2}\frac{(4\pi)^{s-1}}{2\pi}\frac{2^k\Gamma(k)}{\Gamma(s+k-1)} \{I_{k}(\Delta_v,t_v,s) + I_{k}(\Delta_v, -t_v,s)\}.\]
\end{lem}
\begin{proof}
We omit the subscript $v$ in the proof.
First let us consider the case $\Delta=t^2-4m>0$.
By \cite[p.134]{Zagier}, $I_k(\Delta,t,s) + I_k(\Delta,-t,s)$ equals
\[\bigg(\frac{\Delta}{4}\bigg)^{(s-k)/2} 2\cos\left(\frac{\pi}{2}(k-s)\right)
\frac{2^{1-k}\pi \Gamma(k-1+s)\Gamma(k-s)}{\Gamma(k)} 
\left(\frac{t^2}{\Delta}-1\right)^{-(k-1)/2}\mathfrak{P}_{-s}^{1-k}(\tfrac{|t|}{\sqrt{\Delta}})\]
when $1-k<\Re(s)<k$. Invoking the expression
\[\mathfrak{P}_{-s}^{1-k}(\tfrac{|t|}{\sqrt{\Delta}})
=\frac{2^{1-k}(\tfrac{t^2}{\Delta}-1)^{(k-1)/2}}{\Gamma(k)}
{}_2F_{1}(\tfrac{k-s}{2}, \tfrac{k+s-1}{2}; k; \tfrac{4m}{-\Delta})
\]
by \cite[p.156, 7]{MOS},
the right-hand side of the assertion equals
\begin{align*}
	2^{s} m^{(k-1)/2} \bigg(\frac{\pi}{\mathfrak{d}}\bigg)^{s-1}\Delta^{(s-k)/2}
\frac{\Gamma(k-s)}{\Gamma(k)}\cos\left(\frac{\pi}{2}(k-s)\right)
{}_2F_{1}(\tfrac{k-s}{2}, \tfrac{k+s-1}{2}; k ;\tfrac{4m}{-\Delta}).
\end{align*}
This coincides with $I(t, m, s)$ by \cite[(4.3)]{Mizumoto}.

Next let us consider the case $\Delta = t^2-4m<0$.
Noting the choice of a branch of  $\sqrt{z^2-1}$ as in
\cite[\S9.2.1]{SugiyamaTsuzuki2018},
the argument in \cite[p.134]{Zagier} leads us the identity
\begin{align*}I_k(\Delta,t,s) +I_{k}(\Delta, -t,s) 
= & \bigg(\frac{|\Delta|}{4}\bigg)^{(s-k)/2}
\frac{2^{1-k}\pi \Gamma(k-1+s)\Gamma(k-s)}{\Gamma(k)} \times
i^{-k+1}\\
& \times \bigg(\frac{t^2}{-\Delta}+1\bigg)^{-(k-1)/2}\sgn(t) \bigg(\mathfrak{P}_{-s}^{1-k}(\tfrac{it}{\sqrt{|\Delta|}})-\mathfrak{P}_{-s}^{1-k}(-\tfrac{it}{\sqrt{|\Delta|}})\bigg).
\end{align*}
By using \cite[p.126--127 (22)]{Erdelyi} and \cite[p.123, (11)]{Erdelyi},
$\mathfrak{P}_{-s}^{1-k}(ia)$ for $a \in (-1,1)-\{0\}$ is written as
\begin{align*}
\mathfrak{P}_{-s}^{1-k}(ia)= &
\frac{2^{1-k}\pi^{1/2}\sgn(a)i^{k-1}(a^2+1)^{(k-1)/2}}{\Gamma(\tfrac{k+s}{2})\Gamma(\tfrac{1-s+k}{2})}{}_2F_{1}(\tfrac{k-s}{2}, \tfrac{k+s-1}{2}; \tfrac{1}{2};(ia)^2)
\\
&-\frac{\pi^{1/2}2^{2-k}ia\, \sgn(a)i^{k-1}(a^2+1)^{(k-1)/2}}{\Gamma(\tfrac{k-s}{2})\Gamma(\tfrac{k+s-1}{2})}
{}_2F_{1}(\tfrac{k+s}{2}, \tfrac{k-s+1}{2}; \tfrac{3}{2};(ia)^2).
\end{align*}
By substituting $\pm\frac{t}{\sqrt{|\Delta|}}$ for $a$ in this formula when $0<|\frac{t}{\sqrt{|\Delta|}}|<1$, we obtain the identity
\begin{align*}I_k(\Delta,t,s) +I_{k}(\Delta, -t,s)
	=  2^{2-2s}|\Delta|^{(s-k)/2}\pi\frac{\Gamma(k-1+s)\Gamma(\tfrac{k-s}{2})}{\Gamma(k)\Gamma(\tfrac{k+s}{2})}{}_2F_{1}(\tfrac{k-s}{2}, \tfrac{k+s-1}{2}; \tfrac{1}{2}; \tfrac{t^2}{\Delta}).
\end{align*}
We remark that this equality is valid for general $\frac{t}{\sqrt{|\Delta|}}$ by analytic continuation.
As a result, the right-hand side of the assertion equals
\begin{align*}
|\Delta|^{(s-k)/2} m^{(k-1)/2} 2^{k-1}\left(\frac{\pi}{\mathfrak{d}}\right)^{s-1}\frac{\Gamma(\tfrac{k-s}{2})}{\Gamma(\frac{k+s}{2})}
{}_2F_{1}(\tfrac{k-s}{2}, \tfrac{k+s-1}{2}; \tfrac{1}{2}; \tfrac{t^2}{\Delta}).
\end{align*}
This coincides with $I(t,m,s)$ by \cite[(4.3)]{Mizumoto}. Hence we are done.
\end{proof}

Combining \eqref{cuspidal of Mizumoto}, \eqref{unipotent z of Mizumoto}, \eqref{unipotent -z of Mizumoto},
\eqref{hyperbolic elliptic of Mizumoto}, and Lemma \ref{Takase I = Mizumoto I},
the ST trace formula recovers Mizumoto's formula (Theorem \ref{Mizumoto formula}).

\section*{Acknowledgements}
The author would like to thank Masao Tsuzuki for suggestion, fruitful discussion
and careful reading of the early draft.
He also would like to thank Shota Inoue for fruitful discussion on partial summation.
Thanks are also due to Satoshi Wakatsuki for useful comments on low-lying zeros for Hecke-Maass forms.
The author was supported by
Grant-in-Aid for Young Scientists (20K14298).


\end{document}